\documentclass{article}

\usepackage{amssymb}
\usepackage{amsmath}
\usepackage{amsthm}
\usepackage{color}
\usepackage{graphics}
\definecolor{darkgreen}{rgb}{0,0.5,0}

\numberwithin{equation}{section}
\newtheorem{thm}[equation]{\sc Theorem}
\newtheorem{lem}[equation]{\sc Lemma}
\newtheorem{cor}[equation]{\sc Corollary}
\newtheorem{prop}[equation]{\sc Proposition}

\newtheoremstyle{notation}{3pt}{3pt}{}{}{\itshape}{:}{.5em}{\thmname{#1}}
\theoremstyle{notation}

\newtheorem{rem}{\it Remark}

\newtheorem{defin}{\it Definition}
\newtheorem{ex}{\it Example}
\newtheorem{qu}{\it Question}
\newtheorem{alg}{\it Algorithm}

\makeatletter
\renewcommand{\@seccntformat }[1]{\csname the#1\endcsname. }
\makeatother

\chardef\xnearrow='045
\chardef\xnwarrow='055
\chardef\xsearrow='046
\chardef\xswarrow='056

 1

\newcommand\CD{{\mathcal D}}

\newcommand\epv {{$\mbox{}$\hfill ${\Box}$\vspace*{1.5ex} }}

\newcommand\sinfty{\raisebox{1pt}{$\hspace{-1pt}\scriptstyle\infty\hspace{-1pt}$}}

\newcommand\Hom{\mbox{\rm Hom}}

\newcommand\Aut{\mbox{\rm Aut}}

\renewcommand\Im{\mbox{\rm Im}}
\newcommand\Coker{\mbox{\rm Coker}}

\newcommand\barD{\overline{\mathcal D}_{\alpha,\gamma}^\beta}
\newcommand\boxleq{\leq_{\rm box}}
\newcommand\arcleq{\leq_{\rm arc}}

\newcommand\degleq{\leq_{\rm deg}}

\newcommand\vedge[1]{{\buildrel{#1} \over {\hbox to
20pt{\hspace{-0.2em}$-$\hspace{-0.2em}$-$\hspace{-0.2em}$-$ }}}}

\input prepictex   \input pictex    
\input postpictex

\newcounter{boxsize}
\newcounter{tempcounter}
\newcommand{\smallentryformat}{\scriptstyle\sf}
\newcommand\smbox{\put(0,0){\line(1,0){\value{boxsize}}}%
  \put(\value{boxsize},0){\line(0,1){\value{boxsize}}}%
  \put(0,0){\line(0,1){\value{boxsize}}}%
  \put(0,\value{boxsize}){\line(1,0){\value{boxsize}}}}
\newcommand\numbox[1]{\put(0,0)\smbox%
  \put(0,0){\makebox(\value{boxsize},\value{boxsize})[c]{%
      $\smallentryformat#1$}}}
\newcommand\singlebox[1]{\raisebox{-.4ex}{\begin{picture}(4,0)\setcounter{boxsize}{3}%
    \put(0,0)\smbox%
    \put(0,0){\makebox(\value{boxsize},\value{boxsize})[c]{%
      $\scriptstyle\sf#1$}}\end{picture}}}

\newcommand\boxes[2]{\ifthenelse{#2=3}{$\scriptstyle P_2^{#1}$}{%
                                       $\scriptstyle P_{#2}^{#1}$}}

\renewcommand\sb[1]{\makebox[0mm]{$\scriptstyle#1$}}

\begin{document}
\thispagestyle{empty}
\sloppy


\bigskip\bigskip
\begin{center}
{\large\bf Arc diagram varieties} 
\end{center}

\par

\begin{center}
Justyna Kosakowska and Markus Schmidmeier
\footnote{The first named author is partially supported by the research grant No.\ DEC-2011/02/A/ ST1/00216
of the Polish National Science Center.}\vspace{1cm}
\footnote{This work is partially supported by a grant from the 
Simons Foundation (Grant number 245848 to the second named author).}

\centerline{}

\bigskip \parbox{10cm}{\footnotesize{\bf Abstract:}
Let $k$ be an algebraically closed field and $\alpha$, $\beta$, $\gamma$
be partitions.
An algebraic group acts on the constructible set of 
short exact sequences of nilpotent $k$-linear operators 
of Jordan types $\alpha$, $\beta$, and $\gamma$, respectively;
we are interested in the stratification given by the orbits
in the case where all parts of $\alpha$ are at most 2.
Geometric properties of the degeneration relation are controlled by the
combinatorics of arc diagrams. 
We ask if all saturated chains of strata have the same length.
Using arc diagrams we show that this property is not true in general
but holds in case $\beta\setminus\gamma$ is a vertical stripe.
The extended bubble sort algorithm is used to construct chains of 
orbits such that subsequent strata have dimension difference equal to one.
 }

\smallskip\parbox{10cm}{\footnotesize{\bf MSC 2010:} 
Primary: 14L30, 
Secondary: 
05C85,  
16G20, 
47A15,  
68P10  
}

\smallskip\parbox{10cm}{\footnotesize{\bf Key words:} 
degenerations, partial orders, Hall polynomials, 
nilpotent operators, invariant subspaces, Littlewood-Richardson
tableaux, stratification, saturated chains, bubble sort algorithm, 
Bruhat order} 
\end{center}

\medskip

\section{Introduction}

Let $k$ be an algebraically closed field.
For a partition $\alpha=(\alpha_1\geq\ldots\geq\alpha_n)$
we denote by $N_\alpha$ the nilpotent linear operator $T:V\to V$
where $V$ is a $k$-vector space of dimension $|\alpha|=\alpha_1+\cdots+\alpha_n$
and where the operator $T$ can be represented by a matrix 
of Jordan type $\alpha$.
Denote by $\mathcal N$ the category of all  nilpotent linear operators. 
It is well-known that the map $\alpha\mapsto N_\alpha$
defines a one-to-one correspondence between the set 
of all partitions and the set of isomorphism classes of objects in $\mathcal N$
\cite[II,(1.4)]{macd}.

\smallskip
Let $\alpha$, $\beta$, $\gamma$ be partitions.
The affine variety 
$\mathbb H_\alpha^\beta=\Hom_k(N_\alpha,N_\beta)$
(consisting of all $|\beta|\times |\alpha|$-matrices with coefficients in $k$)
with the Zariski topology
contains as constructible subset the set 
$\mathbb V_{\alpha,\gamma}^\beta$ of monomorphisms
$f:N_\alpha\to N_\beta$ such that $\Coker f\cong N_\gamma$.
We consider $\mathbb V_{\alpha,\gamma}^\beta$ as a topological space with the
induced topology.

\smallskip
On $\mathbb V_{\alpha,\gamma}^\beta$ acts the algebraic group 
$G=\Aut_{\mathcal{N}}(N_\alpha)\times \Aut_{\mathcal{N}}( N_\beta)$
via $(g,h)\cdot f=hfg^{-1}$. 
The orbits of this action correspond bijectively to the isomorphism
classes of short exact sequences 
$$0\longrightarrow N_\alpha \longrightarrow N_\beta\longrightarrow N_\gamma
\longrightarrow 0$$
of the given type $(\alpha,\beta,\gamma)$.
In the case where all parts of $\alpha$ are at most 2, the orbits are
in one-to-one correspondence with arc diagrams and define a stratification
for $\mathbb V_{\alpha,\gamma}^\beta$.

\smallskip
In Section~\ref{section-stratification} we compute the orbit dimensions, and
describe in terms of operations on arc diagrams which orbits form
the boundary of a given orbit.  The orbits together with the degeneration
relation form the partially ordered set $\mathcal D_{\alpha,\gamma}^\beta$.
We review results from \cite{kossch} 
and list some references regarding the history of the underlying
counting and isomorphism problems for subgroup embeddings.

\smallskip
In Section~\ref{section-posets} we deal with the question whether all
saturated chains in $\mathcal D_{\alpha,\gamma}^\beta$ have the same length.
While this is not the case in general, we obtain a positive answer
in case $\beta\setminus\gamma$ is a vertical stripe
(Corollaries~\ref{cor-saturated-chains} and \ref{cor-sat-general}).

\smallskip
In this situation, 
the extended bubble sort algorithm in Section~\ref{section-bubble-sort}
produces saturated chains in $\mathcal D_{\alpha,\gamma}^\beta$
such that any two subsequent orbits have dimension difference one. 

\smallskip
In the last Section~\ref{section-excursions} 
we discuss links to projective varieties; in fact, projective spaces
and Grassmannians occur as epimorphic images of arc diagram varieties
of type $\mathbb V_{\alpha,\gamma}^\beta$.  Finally we note that the 
degeneration order for nilpotent operators is just the opposite
order of a natural partial ordering 
for Littlewood-Richardson tableaux (Proposition~\ref{prop-dbar}).

\medskip
{\it Acknowledgement.} The authors would like to thank Birge Huisgen-Zimmermann for
her interest in their work.  In fact, her questions regarding the length
of saturated chains in arc diagram varieties have motivated this paper.

\section{The stratification}\label{section-stratification}

We assume throughout that $k$ is an algebraically closed field, 
and that $\alpha$, $\beta$, $\gamma$ are partitions where 
$\alpha$ is such that
all parts in $\alpha$ are at most $2$, i.e.\ $\alpha_1\leq 2$ holds.
Then the conjugate $\alpha'$ of $\alpha$ has two parts
$\alpha'=(\alpha'_1,\alpha'_2)$ where $\alpha'_2$ counts the number of $2$'s
in $\alpha$ and $\alpha'_1-\alpha'_2$ the number of $1$'s.

\subsection{From short exact sequences to arc diagrams}

\begin{defin}
  \begin{enumerate}
  \item An {\bf arc diagram} $\Delta$ {\bf of type} $(\alpha,\beta,\gamma)$ has 
    $\alpha'_2$ arcs and $\alpha'_1-\alpha'_2$ poles
    which are arranged such that at each point $i$, 
    the number of arcs and poles starting or ending 
    is $\beta'_i-\gamma'_i$.
  \item     By $\mathcal D_{\alpha,\gamma}^\beta$ we denote the set of all
    arc diagrams of type $(\alpha,\beta,\gamma)$.
  \end{enumerate}
\end{defin}

\begin{ex}
Let $\alpha=(2,2,1,1)$, (so $\alpha'=(4,2)$), 
$\beta=(4,3,3,2,2,1)$, $\gamma=(3,2,2,1,1)$.
The following arc diagrams have type $(\alpha,\beta,\gamma)$.
\setlength\unitlength{.8mm}
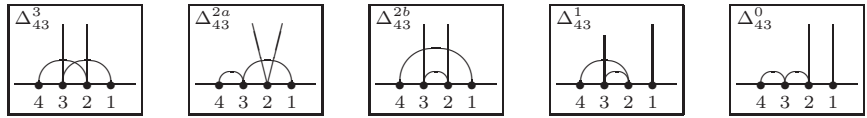
$$
  \begin{picture}(140,20)
    \put(30,0){\framebox(22,18)[t]{%
        \begin{picture}(20,12)(0,5)
          \put(3,14){\sb{\Delta_{43}^{2a}}}
          \put(0,4){\line(1,0){20}}
          \multiput(4,3)(4,0)4{\sb\bullet}
          \put(4,0){\sb 4}
          \put(8,0){\sb 3}
          \put(12,0){\sb 2}
          \put(16,0){\sb 1}
          \put(6,4){\oval(4,4)[t]}
          \put(12,4){\oval(8,8)[t]}
          \put(12,4){\line(1,4){2.5}}
          \put(12,4){\line(-1,4){2.5}}
    \end{picture}}}
    \put(60,0){\framebox(22,18)[t]{%
        \begin{picture}(20,12)(0,5)
          \put(3,14){\sb{\Delta_{43}^{2b}}}
          \put(0,4){\line(1,0){20}}
          \multiput(4,3)(4,0)4{\sb\bullet}
          \put(4,0){\sb 4}
          \put(8,0){\sb 3}
          \put(12,0){\sb 2}
          \put(16,0){\sb 1}
          \put(10,4){\oval(4,4)[t]}
          \put(10,4){\oval(12,12)[t]}
          \put(12,4){\line(0,1){10}}
          \put(8,4){\line(0,1){10}}
    \end{picture}}}
    \put(0,0){\framebox(22,18)[t]{%
        \begin{picture}(20,12)(0,5)
          \put(3,14){\sb{\Delta_{43}^3}}
          \put(0,4){\line(1,0){20}}
          \multiput(4,3)(4,0)4{\sb\bullet}
          \put(4,0){\sb 4}
          \put(8,0){\sb 3}
          \put(12,0){\sb 2}
          \put(16,0){\sb 1}
          \put(8,4){\oval(8,8)[t]}
          \put(12,4){\oval(8,8)[t]}
          \put(12,4){\line(0,1){10}}
          \put(8,4){\line(0,1){10}}
    \end{picture}}}
    \put(90,0){\framebox(22,18)[t]{%
        \begin{picture}(20,12)(0,5)
          \put(3,14){\sb{\Delta_{43}^1}}
          \put(0,4){\line(1,0){20}}
          \multiput(4,3)(4,0)4{\sb\bullet}
          \put(4,0){\sb 4}
          \put(8,0){\sb 3}
          \put(12,0){\sb 2}
          \put(16,0){\sb 1}
          \put(10,4){\oval(4,4)[t]}
          \put(8,4){\oval(8,8)[t]}
          \put(8,4){\line(0,1){8}}
          \put(16,4){\line(0,1){10}}
    \end{picture}}}
    \put(120,0){\framebox(22,18)[t]{%
        \begin{picture}(20,12)(0,5)
          \put(3,14){\sb{\Delta_{43}^0}}
          \put(0,4){\line(1,0){20}}
          \multiput(4,3)(4,0)4{\sb\bullet}
          \put(4,0){\sb 4}
          \put(8,0){\sb 3}
          \put(12,0){\sb 2}
          \put(16,0){\sb 1}
          \put(6,4){\oval(4,4)[t]}
          \put(10,4){\oval(4,4)[t]}
          \put(12,4){\line(0,1){10}}
          \put(16,4){\line(0,1){10}}
    \end{picture}}}
\end{picture}$$

\end{ex}


Before we give a detailed description of the stratification
$\{\mathbb V_\Delta:\Delta\in\mathcal D_{\alpha,\gamma}^\beta\}$ 
for $\mathbb V_{\alpha,\gamma}^\beta$,
we review briefly how tableaux provide a link between short exact sequences
and arc diagrams.

\smallskip
The following result is stated in \cite{klein1} for $p$-groups:

\begin{thm}
Given partitions $\alpha$, $\beta$, $\gamma$, there exists a short exact
sequence of nilpotent linear operators 
$0\to N_\alpha\to N_\beta\to N_\gamma\to 0$ if and only if there exists
a Littlewood-Richardson (LR-) tableau $\Gamma$ of type $(\alpha,\beta,\gamma)$.
\end{thm}

\begin{defin}
\begin{enumerate}\item
    Given three partitions $\alpha,\beta,\gamma$, 
    an {\bf LR-tableau} of type $(\alpha,\beta,\gamma)$ is a skew diagram 
    of shape $\beta\backslash \gamma$ 
    with $\alpha'_1$ entries $\singlebox 1$,
    $\alpha'_2$ entries $\singlebox 2$, etc.
    The entries are weakly increasing in each row, strictly
    increasing in each column, and satisfy the lattice permutation property 
    (for each $c\geq 0$, $\ell\geq 2$ there are at least 
    as many entries $\ell-1$ on the 
    right hand side of the $c$-th column as there are entries $\ell$).
\item    The {\bf LR-coefficient} $c_{\alpha,\gamma}^\beta$ counts the 
    number of LR-tableaux of type $(\alpha,\beta,\gamma)$.
\end{enumerate}
\end{defin}

\begin{ex}\label{ex-gammas}
Let $\alpha=(2,2,1,1)$, $\beta=(4,3,3, 2,2,1)$, $\gamma=(3,2,2,1,1)$.
There are 4 LR-tableaux of type $(\alpha,\beta,\gamma)$,
so $c_{\alpha,\gamma}^\beta=4$.

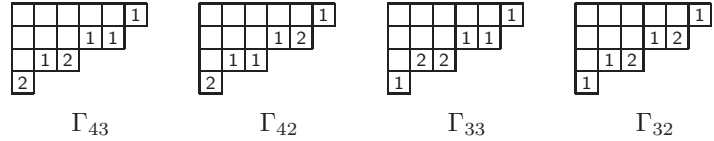
\setlength\unitlength{1mm}
$$
\begin{picture}(95,20)(0,-2)
  \put(0,10){
    \begin{picture}(18,12)(0,6)
      \put(8,-5){$\Gamma_{43}$}
      \multiput(0,9)(3,0)5{\smbox}
      \put(15,9){\numbox{1}}
      \multiput(0,6)(3,0)3{\smbox}
      \put(12,6){\numbox{1}}
      \put(9,6){\numbox{1}}
      \put(0,3){\smbox}
      \put(3,3){\numbox{1}}
      \put(6,3){\numbox{2}}
      \put(0,0){\numbox{2}}
    \end{picture}
  }
  \put(25,10){
    \begin{picture}(18,12)(0,6)
      \put(8,-5){$\Gamma_{42}$}
      \multiput(0,9)(3,0)5{\smbox}
      \put(15,9){\numbox{1}}
      \multiput(0,6)(3,0)3{\smbox}
      \put(12,6){\numbox{2}}
      \put(9,6){\numbox{1}}
      \put(0,3){\smbox}
      \put(3,3){\numbox{1}}
      \put(6,3){\numbox{1}}
      \put(0,0){\numbox{2}}
    \end{picture}
  }
  \put(50,10){
    \begin{picture}(18,12)(0,6)
      \put(8,-5){$\Gamma_{33}$}
      \multiput(0,9)(3,0)5{\smbox}
      \put(15,9){\numbox{1}}
      \multiput(0,6)(3,0)3{\smbox}
      \put(12,6){\numbox{1}}
      \put(9,6){\numbox{1}}
      \put(0,3){\smbox}
      \put(3,3){\numbox{2}}
      \put(6,3){\numbox{2}}
      \put(0,0){\numbox{1}}
    \end{picture}
  }
  \put(75,10){
    \begin{picture}(18,12)(0,6)
      \put(8,-5){$\Gamma_{32}$}
      \multiput(0,9)(3,0)5{\smbox}
      \put(15,9){\numbox{1}}
      \multiput(0,6)(3,0)3{\smbox}
      \put(12,6){\numbox{2}}
      \put(9,6){\numbox{1}}
      \put(0,3){\smbox}
      \put(3,3){\numbox{1}}
      \put(6,3){\numbox{2}}
      \put(0,0){\numbox{1}}
    \end{picture}
  }
\end{picture}
$$
(In the expression $\Gamma_{ij}$, the subscript $ij$ lists the rows which
contain the symbol $\singlebox2$, and hence determines the LR-tableau
uniquely in the case where $\alpha_1\leq 2$.)
\end{ex}

\begin{defin}
  \begin{enumerate}
  \item     A {\bf Klein tableau} 
    of type $(\alpha,\beta,\gamma)$
    is a refinement of the LR-tableau of the same type
    in the sense that each entry $\ell\geq 2$ carries a subscript,
    subject to the following conditions (see \cite[(1.2)]{klein}):
    \begin{enumerate}
    \item If a symbol $\singlebox {\ell_r}$ 
      occurs in the $m$-th row in the tableau, 
      then $1\leq r\leq m-1$.
    \item If $\singlebox{\ell_r}$ occurs in the $m$-th row
      and the entry above $\singlebox{\ell_r}$ is $\ell-1$, 
      then $r=m-1$.
    \item The total number of symbols $\singlebox{\ell_r}$ in the tableau 
      cannot exceed the number of entries $\ell-1$ in row $r$. 
    \end{enumerate}
  \item 
    Let $\Gamma$ be an LR-tableau with entries at most 2, and $\Pi$ a Klein
    tableau which refines $\Gamma$.  The {\bf arc diagram corresponding to}
    $\Pi$ is obtained by drawing an arc from $m$ to $j$ for each pair of boxes
    $\singlebox{2_j}$ in row $m$ and $\singlebox{1}$ in row $j$,
    and by drawing a pole at $r$ 
    for each remaining box $\singlebox{1}$ in row $r$.
  \end{enumerate}
\end{defin}

\begin{ex}
Here are the five Klein tableaux
which refine the LR-tableau $\Gamma_{43}$.
\nopagebreak

\setlength\unitlength{1mm}
$$
\begin{picture}(120,16)(0,-2)
  \put(0,10){
    \begin{picture}(18,12)(0,6)
      \put(8,-5){$\Pi_{43}^3$}
      \multiput(0,9)(3,0)5{\smbox}
      \put(15,9){\numbox{1}}
      \multiput(0,6)(3,0)3{\smbox}
      \put(12,6){\numbox{1}}
      \put(9,6){\numbox{1}}
      \put(0,3){\smbox}
      \put(3,3){\numbox{1}}
      \put(6,3){\numbox{2_1}}
      \put(0,0){\numbox{2_2}}
    \end{picture}
  }
  \put(25,10){
    \begin{picture}(18,12)(0,6)
      \put(8,-5){$\Pi_{43}^{2a}$}
      \multiput(0,9)(3,0)5{\smbox}
      \put(15,9){\numbox{1}}
      \multiput(0,6)(3,0)3{\smbox}
      \put(12,6){\numbox{1}}
      \put(9,6){\numbox{1}}
      \put(0,3){\smbox}
      \put(3,3){\numbox{1}}
      \put(6,3){\numbox{2_1}}
      \put(0,0){\numbox{2_3}}
    \end{picture}
  }
  \put(50,10){
    \begin{picture}(18,12)(0,6)
      \put(8,-5){$\Pi_{43}^{2b}$}
      \multiput(0,9)(3,0)5{\smbox}
      \put(15,9){\numbox{1}}
      \multiput(0,6)(3,0)3{\smbox}
      \put(12,6){\numbox{1}}
      \put(9,6){\numbox{1}}
      \put(0,3){\smbox}
      \put(3,3){\numbox{1}}
      \put(6,3){\numbox{2_2}}
      \put(0,0){\numbox{2_1}}
    \end{picture}
  }
  \put(75,10){
    \begin{picture}(18,12)(0,6)
      \put(8,-5){$\Pi_{43}^1$}
      \multiput(0,9)(3,0)5{\smbox}
      \put(15,9){\numbox{1}}
      \multiput(0,6)(3,0)3{\smbox}
      \put(12,6){\numbox{1}}
      \put(9,6){\numbox{1}}
      \put(0,3){\smbox}
      \put(3,3){\numbox{1}}
      \put(6,3){\numbox{2_2}}
      \put(0,0){\numbox{2_2}}
    \end{picture}
  }
  \put(100,10){
    \begin{picture}(18,12)(0,6)
      \put(8,-5){$\Pi_{43}^0$}
      \multiput(0,9)(3,0)5{\smbox}
      \put(15,9){\numbox{1}}
      \multiput(0,6)(3,0)3{\smbox}
      \put(12,6){\numbox{1}}
      \put(9,6){\numbox{1}}
      \put(0,3){\smbox}
      \put(3,3){\numbox{1}}
      \put(6,3){\numbox{2_2}}
      \put(0,0){\numbox{2_3}}
    \end{picture}
  }
\end{picture}
$$

The arc diagrams $\Delta_{43}^3$, $\Delta_{43}^{2a}$, $\Delta_{43}^{2b}$,
$\Delta_{43}^1$ and $\Delta_{43}^0$ given by the above Klein tableaux are pictured 
in the beginning of this section.  (The exponent $x$ in $\Delta_{ij}^x$
counts the number of intersections.)
\end{ex}

\smallskip
By $\mathcal S_2$ we denote the category of all sequences
$0\to N_\alpha\to N_\beta\to N_\gamma\to 0$ where $\alpha$ has all parts at most 2.
Based on the classification of the indecomposable embeddings in
\cite[Theorem~7.5]{bhw}, it is shown in \cite[Proposition~2]{sch} 
that there is a one-to-one correspondence
$$\big\{\text{objects in $\mathcal S_2$}\big\}{\big/}_{\cong} \;
\stackrel{\text{1-1}}\longleftrightarrow\;
\big\{\text{Klein tableaux with entries at most 2}\big\}.$$

Summarizing we obtain:

\begin{cor}\label{lemma-Gorbits}
  There is a one-to-one correspondence
  $$\mathcal D_{\alpha,\gamma}^\beta \stackrel{\text{\rm 1-1}}\longleftrightarrow
  \big\{G\text{-orbits in }\mathbb V_{\alpha,\gamma}^\beta\big\}.$$
\end{cor}

\subsection{Strata given by arc diagrams}

\begin{defin}
For an arc diagram $\Delta$ of type $(\alpha,\beta,\gamma)$ 
we denote the corresponding $G$-orbit
in $\mathbb V_{\alpha,\gamma}^\beta$ by $\mathbb V_\Delta$.
\end{defin}

\begin{prop}
  The set $\{\mathbb V_\Delta:\Delta\in\mathcal D_{\alpha,\gamma}^\beta\}$
  forms a stratification for $\mathbb V=\mathbb V_{\alpha,\gamma}^\beta$ 
  in the sense that
  \begin{enumerate}
  \item  Each $\mathbb V_\Delta$ is locally closed in $\mathbb V$.
  \item  $\mathbb V$ is the disjoint union 
    $\bigcup^\bullet_{\Delta\in\mathcal D_{\alpha,\gamma}^\beta}\mathbb V_\Delta$.
  \item For each $\Delta$ there is a finite subset
    $U_\Delta \subset \mathcal D_{\alpha,\gamma}^\beta$
    such that the closure $\overline{\mathbb V}_\Delta$ 
    is just the union $\bigcup_{\Gamma\in U_\Delta} \mathbb V_\Gamma$.
  \end{enumerate}  
\end{prop}

\begin{proof}
We have seen in Corollary~\ref{lemma-Gorbits} that $\mathbb V_{\alpha,\gamma}^\beta$
is the finite union of the $G$-orbits of type $\mathbb V_\Delta$. 
According to \cite[Proposition~8.3]{humphreys}, each orbit is a 
smooth and locally closed subset of $\mathbb V_{\alpha,\gamma}^\beta$ 
whose boundary is a union of orbits of strictly lower dimension.
\end{proof}

\begin{rem}
\begin{enumerate}
\item
  The condition on the field $k$ to be algebraically closed is only 
  needed for the last statement.  Otherwise, the field can be arbitrary,
  in fact, there need not even be a field:
  For  $\Lambda$  a discrete valuation domain with maximal ideal $m$, 
  we can define    $N_\alpha(\Lambda) =\bigoplus_{i=1}^s \Lambda/(m^{\alpha_i})$.
  In particular if $\Lambda$ is the localization $\mathbb Z_p$, then
  we are dealing with finite abelian $p$-groups.
\item 
  The problem of classifying the orbits in $\mathbb V_{\alpha,\gamma}^\beta$
  has been posed by G.\ Birkhoff in 1934 \cite{birkhoff} for
  $\Lambda=\mathbb Z_p$:  Classify all subgroups $A$ 
  of a finite abelian $p$-group
  $B$, up to automorphisms of $B$. 
  In general, the problem is considered infeasible, 
  see for example \cite{rs-wild}, but there are many partial 
  and related results:
  If the exponent of $B$ is at most $5$, then the category of embeddings
  has finite type \cite{richman-w};
  for $\Lambda=k[T]_{(T)}$, tame type occurs if the exponent of $B$ is
  at most $6$ \cite{rs};
  our category $\mathcal S_2$ has discrete representation type \cite{bhw};
  for the related problem of studying lattices over tiled orders we
  refer to  \cite{rump};
  categories of embeddings of graded operators occur in singularity
  theory \cite{klm};
  for a classification of the representation types of chain categories
  we refer to \cite{simson};
  please see \cite{zhang} 
  for homological properties of categories of embeddings.
\end{enumerate}
\end{rem}

\subsection{The dimensions of the strata}

\medskip
In this subsection we review the dimension formula:

\begin{prop}
Let $\Delta$ be an arc diagram of type $(\alpha,\beta,\gamma)$.
The stratum $\mathbb V_\Delta$ is a smooth irreducible variety 
of dimension
$$\dim \mathbb V_\Delta = \deg g_{\alpha,\gamma}^\beta + \deg a_{\alpha} - x(\Delta).$$
\end{prop}

First we define the terms in the dimension formula.
From results given in \cite{klein} one can deduce the following theorem.

\begin{thm}
 For any partition triple $(\alpha,\beta,\gamma)$ there exists a~polynomial
$g_{\alpha,\gamma}^\beta(t)\in \mathbb{Z}[t]$ 
such that for any finite field $k$ we have
$$ g_{\alpha,\gamma}^\beta(|k|)=|\mathbb V_{\alpha,\gamma}^\beta(k)|,$$
where $|X|$ denotes the cardinality of the finite set $X$.
\end{thm}

Polynomials $g_{\alpha,\gamma}^\beta(t)$ are called {\bf Hall polynomials}.
It is known (see \cite{macd}) that 
$$ \deg\, g_{\alpha,\gamma}^\beta(t)=n(\beta)-n(\alpha)-n(\gamma), $$
where for a~partition $\lambda$ the  moment is defined as 
$$n(\lambda)=\sum_{i\geq 0}(i-1)\lambda_i. $$

\smallskip
A formula for the cardinality 
$a_\alpha(q)=|\Aut N_\alpha(\mathbb F_q)|$ 
of the automorphism group of $N_\alpha$ is given in \cite[II, (1.6)]{macd}.
In particular, $\deg a_\alpha = |\alpha|+2 n(\alpha)$.

\smallskip
For an arc diagram $\Delta$, we denote by $x(\Delta)$ the
number of intersections in $\Delta$.

\begin{defin}
For a Littlewood-Richardson tableau $\Gamma$ of type $(\alpha,\beta,\gamma)$,
we say an arc diagram $\Delta$ has {\bf Littlewood-Richardson type} $\Gamma$
if for each $i$, the number of arcs in $\Delta$ starting at $i$
equals the number of $2$'s in the $i$-th row of $\Gamma$.
We write $\mathbb V_\Gamma=\bigcup_{\text{$\Delta$ has type $\Gamma$}}\mathbb V_\Delta$.
\end{defin}

It follows from the previous section:
\begin{equation}
 \mathbb V_{\alpha,\gamma}^\beta\;=\;  \bigcup_\Gamma \; \mathbb V_\Gamma \;=\; 
\bigcup_\Gamma\bigcup_\Delta \; \mathbb V_\Delta,
\label{eq-variety-decomp}\end{equation}
where the first union is indexed by all LR-tableaux $\Gamma$
of type $(\alpha,\beta,\gamma)$ and the second
union is indexed by all arc diagrams $\Delta$ of type $\Gamma$.

\smallskip It is well known that orbits of an~algebraic group action
are locally closed sets. It follows that $\mathbb V_{\alpha,\gamma}^\beta(k)$ and $\mathbb V_\Gamma(k)$
are constructible sets, because they are finite unions of locally closed
sets  $\mathbb V_\Delta(k)$.

\smallskip
Correspondingly, if $k$ is a finite field of $q$ elements, there is the following
sum formula for Hall polynomials,
\begin{equation}
g_{\alpha,\gamma}^\beta(q) \;=\;  \sum_\Gamma g_\Gamma(q) \;=\; 
\sum_\Gamma\sum_{\Delta}g_\Delta(q), 
\label{eq-hall-poly}\end{equation}
where the indices are as above.  
The polynomials $g_\Gamma$ are monic of the same degree
$n(\beta)-n(\alpha)-n(\gamma)$, 
while the polynomials $g_\Delta$ are monic of degree
$n(\beta)-n(\alpha)-n(\gamma)-x(\Delta)$
(\cite[Corollaries~1-3]{klein}, here $x(\Delta)$ 
is the deviation from dominance of the prototype given by
the arc diagram $\Delta$).

\smallskip  The formulae (\ref{eq-variety-decomp}) and (\ref{eq-hall-poly}) 
have a~different nature:
the first one is geometric and the second one is combinatorial. 
The following remarks show that they are ``compatible''.

\smallskip
Results presented in \cite[Section 5]{kossch} give us the following formulae for variety dimensions.
\begin{itemize}
 \item $\dim\, \mathbb V_{\alpha,\gamma}^\beta\;=\;
     \deg\,g_{\alpha,\gamma}^\beta+\deg\,a_\alpha$,
 \item $\dim\, \mathbb V_\Gamma\;=\;\deg\,g_\Gamma + \deg\,a_\alpha\;=\;
   \dim \, \mathbb V_{\alpha,\gamma}^\beta$,
 \item $\dim\, \mathbb V_\Delta\; =\;\deg\,g_\Delta + \deg \,a_\alpha \;=\;
   \dim \,\mathbb V_{\alpha,\gamma}^\beta-x(\Delta)$
\end{itemize}

\begin{rem}
Polynomials and algebras, that we call Hall polynomials and Hall algebras,
where defined and investigated in 1900 by E. Steinitz.
He described their connections with Schur functions.
 However, the results of Steinitz were forgotten.
In the nineteen fifties, Hall polynomials and algebras were defined by 
P.\ Hall  for finite abelian $p$-groups.
In \cite{hall}, P.\ Hall gave only a summary of this theory. 
His work was continued by J.\ A.\ Green \cite{green}
and T.\ Klein \cite{klein}. The reader is referred to \cite{macd}
for more information about Hall polynomials and algebras and for their
connections with symmetric functions.
\end{rem}

\bigskip
\subsection{Geometric properties of $\mathbb V_{\alpha,\gamma}^\beta$}

\begin{defin}
Two diagrams of arcs and poles are said to be in {\bf arc order}
if the first is obtained from the second by a sequence of moves of type 
(A), (B), (C), or (D):

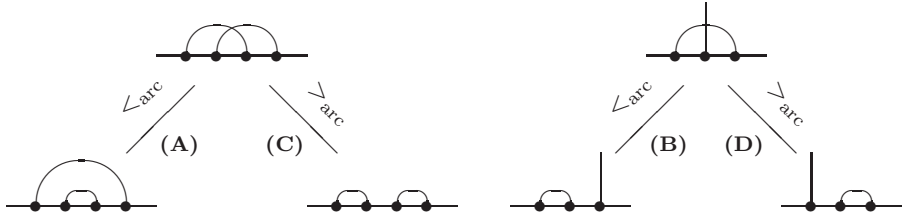
$$
\begin{picture}(60,28)(0,-5)
  \put(0,0){\begin{picture}(20,8)(0,5)
      \put(0,4){\line(1,0){20}}
      \multiput(3,3)(4,0)4{$\bullet$}
      \put(10,4){\oval(4,4)[t]}
      \put(10,4){\oval(12,12)[t]}
    \end{picture}
  }
  \put(20,20){\begin{picture}(20,8)(0,5)
      \put(0,4){\line(1,0){20}}
      \multiput(3,3)(4,0)4{$\bullet$}
      \put(8,4){\oval(8,8)[t]}
      \put(12,4){\oval(8,8)[t]}
    \end{picture}
  }
  \put(25,15){\line(-1,-1){9}}
  \put(23,7){\makebox(0,0){\bf\footnotesize (A)}}
  \put(17,13){\rotatebox{45}{\makebox[0mm]{$<_{\rm arc}$}}}
  \put(35,15){\line(1,-1){9}}
  \put(37,7){\makebox(0,0){\bf\footnotesize(C)}}
  \put(42,12){\rotatebox{-45}{\makebox[0mm]{$>_{\rm arc}$}}}
  \put(40,0){\begin{picture}(20,8)(0,5)
      \put(0,4){\line(1,0){20}}
      \multiput(3,3)(4,0)4{$\bullet$}
      \put(6,4){\oval(4,4)[t]}
      \put(14,4){\oval(4,4)[t]}
    \end{picture}
  }
\end{picture}
\qquad
\begin{picture}(52,28)(0,-5)
  \put(0,0){\begin{picture}(16,8)(0,5)
      \put(0,4){\line(1,0){16}}
      \multiput(3,3)(4,0)3{$\bullet$}
      \put(6,4){\oval(4,4)[t]}
      \put(12,4){\line(0,1){7}}
    \end{picture}
  }
  \put(18,20){\begin{picture}(16,8)(0,5)
      \put(0,4){\line(1,0){16}}
      \multiput(3,3)(4,0)3{$\bullet$}
      \put(8,4){\oval(8,8)[t]}
      \put(8,4){\line(0,1){7}}
    \end{picture}
  }
  \put(36,0){\begin{picture}(16,8)(0,5)
      \put(0,4){\line(1,0){16}}
      \multiput(3,3)(4,0)3{$\bullet$}
      \put(10,4){\oval(4,4)[t]}
      \put(4,4){\line(0,1){7}}
    \end{picture}
  }
  \put(23,15){\line(-1,-1){9}}
  \put(21,7){\makebox(0,0){\bf\footnotesize (B)}}
  \put(15,13){\rotatebox{45}{\makebox[0mm]{$<_{\rm arc}$}}}
  \put(29,15){\line(1,-1){9}}
  \put(31,7){\makebox(0,0){\bf\footnotesize (D)}}
  \put(36,12){\rotatebox{-45}{\makebox[0mm]{$>_{\rm arc}$}}}
\end{picture}
$$
If the arc diagrams $\Delta$ and $\Delta'$ are in relation, 
we write $\Delta\arcleq \Delta'$.
\end{defin}

\medskip
The main result in \cite{kossch} states that the arc order 
and the degeneration order on arc diagrams are related:

\begin{thm}\label{thm-first-main}
Suppose that $k$ is an algebraically closed field and that $\alpha,\beta,\gamma$
are partitions with $\alpha_1\leq 2$. 
For arc diagrams $\Delta$, $\Delta'$ of type $(\alpha,\beta,\gamma)$
we have
$$\Delta\degleq \Delta' \qquad\text{if and only if}
\qquad \Delta \arcleq \Delta'$$
 where by definition $\Delta\degleq \Delta'$ if and only if
 $\mathbb{V}_{\Delta'}\subseteq \overline{\mathbb V}_\Delta$.
\end{thm}

The Littlewood-Richardson coefficient $c_{\alpha,\gamma}^\beta$ counts the
number of LR-tableaux 
$\Gamma$ of type  $(\alpha,\beta,\gamma)$, see \cite{macd}.
It follows that in the sum (\ref{eq-hall-poly}) 
there exist exactly $c_{\alpha,\gamma}^\beta$ polynomials $g_\Gamma$ 
of degree $n_\beta-n_\alpha-n_\gamma$.  

\smallskip
Geometrically it means that there exist 
$c_{\alpha,\gamma}^\beta$ subsets 
$\mathbb V_\Gamma\subseteq \mathbb V_{\alpha,\gamma}^\beta$ with the maximal
dimension $n_\beta-n_\alpha-n_\gamma+\deg\,a_\alpha$. 
Moreover, for any such a~subset $\mathbb V_\Gamma$ we have
$$\dim\, \mathbb V_\Gamma=n_\beta-n_\alpha-n_\gamma+\deg\,a_\alpha
  =\dim\, \mathbb V_{\alpha,\gamma}^\beta$$
and $\mathbb V_{\alpha,\gamma}^\beta=\bigcup_\Gamma \overline{\mathbb V}_\Gamma,$
where the union runs over all LR-tableaux $\Gamma$ with maximal dimension.
As a~consequence we get the following fact.
\begin{lem}
\begin{enumerate}
\item  $\mathbb V_{\alpha,\gamma}^\beta$ is irreducible if and only if
  $c_{\alpha,\gamma}^\beta=1$. 
\item There exists $c_{\alpha,\gamma}^\beta$ 
  irreducible components of $\mathbb V_{\alpha,\gamma}^\beta$.
\end{enumerate}
\end{lem}

\medskip
For any LR-tableau $\Gamma$ there exits exactly one arc diagram
with no intersections, see \cite{kossch}. 
This diagram $\Delta$ satisfies
   $\overline{\mathbb V}_\Delta\cap\mathbb V_\Gamma = \mathbb V_\Gamma$.  
We deduce the following fact. 

\begin{lem}
$\mathbb V_\Gamma$ is an~irreducible set. 
\end{lem}

Moreover $\mathbb V_{\alpha,\gamma}^\beta=\bigcup_\Delta \overline{\mathbb V}_\Delta,$
where the union runs over all arc diagrams with no intersections.

\section{Partially ordered sets}\label{section-posets}

\medskip

\medskip
Let $\mathcal D_\Gamma$ 
be the set of all arc diagrams given by an~LR-tableau $\Gamma$ with entries
at most two,
  and $\mathcal D_{\alpha,\gamma}^\beta $ 
the set of all arc diagrams of partition type
  $(\alpha,\beta,\gamma)$ with $\alpha_1\leq 2$.
We describe properties of the posets $\CD_\Gamma=(\CD_\Gamma,\leq_{\rm arc})$ 
and $\CD_{\alpha,\gamma}^\beta=(\CD_{\alpha,\gamma}^\beta,\leq_{\rm arc})$.

\smallskip
In \cite{kossch} the following theorem is proved. 

\begin{thm}\label{thm-lattice} Let $\Gamma$ be an~LR tableau with entries
  at most two.
  \begin{enumerate}
  \item In the poset $\CD_\Gamma$ there exists exactly one minimal element: 
    the arc diagram with no intersections.
  \item In the poset $\CD_\Gamma$ there exists exactly one maximal element,
    the arc diagram with the maximal number of intersections.
  \item In the poset $\CD_{\alpha,\gamma}^\beta$ 
    there exists exactly one maximal element, given by the unique
    arc diagram with the largest number of intersections.
  \item The set of minimal elements of the poset $\CD_{\alpha,\gamma}^\beta$ 
    consists of the intersection-free arc diagrams,
    they are in one-to-one correspondence with the LR-tableaux of type
    $(\alpha,\beta,\gamma)$.
  \end{enumerate} 
\end{thm}

\begin{cor}
  \begin{enumerate}
  \item The open strata in 
    $\{\mathbb V_\Delta:\Delta\in \mathcal D_{\alpha,\gamma}^\beta\}$ 
    are the strata of maximal dimension.  The number of such strata
    is $c_{\alpha,\gamma}^\beta$, they are in one-to-one
    correspondence with the intersection-free arc diagrams.
  \item There is a unique closed stratum, it has minimal dimension
    and is given by the unique arc diagram with the maximal number
    of intersections.
  \end{enumerate}
\end{cor}

By identifying the points in $\mathcal D_{\alpha,\gamma}^\beta$ which correspond
to the same LR-tableau, we obtain the coarser poset 
$\overline{\mathcal D}_{\alpha,\gamma}^\beta$ on the set of LR-tableaux of type
$(\alpha,\beta,\gamma)$.
Proposition~\ref{prop-dbar} below shows that there are several equivalent
candidates for a partial ordering on $\barD$.


\subsection{Two questions about saturated chains}

\begin{defin}
A chain in a poset is {\bf saturated} if it has no refinement.
\end{defin}

\begin{qu}
In $\mathcal D_{\alpha,\gamma}^\beta$, do all saturated chains 
have the same length?
\end{qu}

\begin{qu}[Question 2:]
Suppose $\Delta<\Delta'$ in $\mathcal D_{\alpha,\gamma}^\beta$. 
Is there a chain from $\Delta$ to $\Delta'$ 
where subsequent strata have dimension difference equal to one?
\end{qu}

The example on page~\pageref{figure2}
shows that the answer to both questions is NO.
Take $\alpha=(2,2,1,1)$, $\beta=(4,3,3,2,2,1)$, and $\gamma=(3,2,2,1,1)$.

\smallskip
Consider the five
arc diagrams labelled $\Delta_{43}^x$. 
There are two saturated chains from $\Delta_{43}^0$ to $\Delta_{43}^{3}$,
they have length 2 and 3, respectively. 
Similarly, there are two saturated chains from $\Delta_{32}^0$ to 
$\Delta_{43}^3$, also of length 2 and 3, respectively.
Note that for each direct predecessor of $\Delta_{43}^{2a}$ and
$\Delta_{42}^2$,
the dimension of the corresponding stratum decreases by two.

\bigskip
In the next sections we will obtain an affirmative answer 
to both questions in the case where  $\beta\setminus\gamma$
is a vertical strip (which excludes double poles in any of the arc diagrams).

\setlength\unitlength{.8mm}
\begin{figure}[t!]\label{figure2}
  $$
  \begin{picture}(110,170)
    \put(0,15){\framebox(22,18)[t]{%
        \begin{picture}(20,12)(0,5)
          \put(3,14){\sb{\Delta_{43}^0}}
          \put(0,4){\line(1,0){20}}
          \multiput(4,3)(4,0)4{\sb\bullet}
          \put(4,0){\sb 4}
          \put(8,0){\sb 3}
          \put(12,0){\sb 2}
          \put(16,0){\sb 1}
          \put(6,4){\oval(4,4)[t]}
          \put(10,4){\oval(4,4)[t]}
          \put(12,4){\line(0,1){10}}
          \put(16,4){\line(0,1){10}}
    \end{picture}}}
    \put(43,2){\framebox(22,18)[t]{%
        \begin{picture}(20,12)(0,5)
          \put(3,14){\sb{\Delta_{42}^0}}
          \put(0,4){\line(1,0){20}}
          \multiput(4,3)(4,0)4{\sb\bullet}
          \put(4,0){\sb 4}
          \put(8,0){\sb 3}
          \put(12,0){\sb 2}
          \put(16,0){\sb 1}
          \put(6,4){\oval(4,4)[t]}
          \put(14,4){\oval(4,4)[t]}
          \put(8,4){\line(0,1){10}}
          \put(12,4){\line(0,1){10}}
    \end{picture}}}
    \put(47,28){\framebox(22,18)[t]{%
        \begin{picture}(20,12)(0,5)
          \put(3,14){\sb{\Delta_{33}^0}}
          \put(0,4){\line(1,0){20}}
          \multiput(4,3)(4,0)4{\sb\bullet}
          \put(4,0){\sb 4}
          \put(8,0){\sb 3}
          \put(12,0){\sb 2}
          \put(16,0){\sb 1}
          \put(10,4){\oval(4,4)[t]}
          \put(10,4){\oval(4,6)[t]}
          \put(4,4){\line(0,1){8}}
          \put(16,4){\line(0,1){10}}
    \end{picture}}}
    \put(90,15){\framebox(22,18)[t]{%
        \begin{picture}(20,12)(0,5)
          \put(3,14){\sb{\Delta_{32}^0}}
          \put(0,4){\line(1,0){20}}
          \multiput(4,3)(4,0)4{\sb\bullet}
          \put(4,0){\sb 4}
          \put(8,0){\sb 3}
          \put(12,0){\sb 2}
          \put(16,0){\sb 1}
          \put(14,4){\oval(4,4)[t]}
          \put(10,4){\oval(4,4)[t]}
          \put(4,4){\line(0,1){8}}
          \put(8,4){\line(0,1){10}}
    \end{picture}}}
    \put(25,70){\framebox(22,18)[t]{%
        \begin{picture}(20,12)(0,5)
          \put(3,14){\sb{\Delta_{43}^1}}
          \put(0,4){\line(1,0){20}}
          \multiput(4,3)(4,0)4{\sb\bullet}
          \put(4,0){\sb 4}
          \put(8,0){\sb 3}
          \put(12,0){\sb 2}
          \put(16,0){\sb 1}
          \put(10,4){\oval(4,4)[t]}
          \put(8,4){\oval(8,8)[t]}
          \put(8,4){\line(0,1){8}}
          \put(16,4){\line(0,1){10}}
    \end{picture}}}
    \put(65,70){\framebox(22,18)[t]{%
        \begin{picture}(20,12)(0,5)
          \put(3,14){\sb{\Delta_{33}^1}}
          \put(0,4){\line(1,0){20}}
          \multiput(4,3)(4,0)4{\sb\bullet}
          \put(4,0){\sb 4}
          \put(8,0){\sb 3}
          \put(12,0){\sb 2}
          \put(16,0){\sb 1}
          \put(10,4){\oval(4,4)[t]}
          \put(12,4){\oval(8,8)[t]}
          \put(4,4){\line(0,1){8}}
          \put(12,4){\line(0,1){10}}
    \end{picture}}}
    \put(45,110){\framebox(22,18)[t]{%
        \begin{picture}(20,12)(0,5)
          \put(3,14){\sb{\Delta_{43}^{2b}}}
          \put(0,4){\line(1,0){20}}
          \multiput(4,3)(4,0)4{\sb\bullet}
          \put(4,0){\sb 4}
          \put(8,0){\sb 3}
          \put(12,0){\sb 2}
          \put(16,0){\sb 1}
          \put(10,4){\oval(4,4)[t]}
          \put(10,4){\oval(12,12)[t]}
          \put(12,4){\line(0,1){10}}
          \put(8,4){\line(0,1){10}}
    \end{picture}}}
    \put(0,110){\framebox(22,18)[t]{%
        \begin{picture}(20,12)(0,5)
          \put(3,14){\sb{\Delta_{43}^{2a}}}
          \put(0,4){\line(1,0){20}}
          \multiput(4,3)(4,0)4{\sb\bullet}
          \put(4,0){\sb 4}
          \put(8,0){\sb 3}
          \put(12,0){\sb 2}
          \put(16,0){\sb 1}
          \put(6,4){\oval(4,4)[t]}
          \put(12,4){\oval(8,8)[t]}
          \put(12,4){\line(1,4){2.5}}
          \put(12,4){\line(-1,4){2.5}}
    \end{picture}}}
    \put(90,110){\framebox(22,18)[t]{%
        \begin{picture}(20,12)(0,5)
          \put(3,14){\sb{\Delta_{42}^2}}
          \put(0,4){\line(1,0){20}}
          \multiput(4,3)(4,0)4{\sb\bullet}
          \put(4,0){\sb 4}
          \put(8,0){\sb 3}
          \put(12,0){\sb 2}
          \put(16,0){\sb 1}
          \put(14,4){\oval(4,4)[t]}
          \put(8,4){\oval(8,8)[t]}
          \put(8,4){\line(1,4){2.5}}
          \put(8,4){\line(-1,4){2.0}}
    \end{picture}}}
    \put(45,150){\framebox(22,18)[t]{%
        \begin{picture}(20,12)(0,5)
          \put(3,14){\sb{\Delta_{43}^3}}
          \put(0,4){\line(1,0){20}}
          \multiput(4,3)(4,0)4{\sb\bullet}
          \put(4,0){\sb 4}
          \put(8,0){\sb 3}
          \put(12,0){\sb 2}
          \put(16,0){\sb 1}
          \put(8,4){\oval(8,8)[t]}
          \put(12,4){\oval(8,8)[t]}
          \put(12,4){\line(0,1){10}}
          \put(8,4){\line(0,1){10}}
    \end{picture}}}
    \put(0,15){\framebox(22,18)[t]{%
        \begin{picture}(20,12)(0,5)
          \put(3,14){\sb{\Delta_{43}^0}}
          \put(0,4){\line(1,0){20}}
          \multiput(4,3)(4,0)4{\sb\bullet}
          \put(4,0){\sb 4}
          \put(8,0){\sb 3}
          \put(12,0){\sb 2}
          \put(16,0){\sb 1}
          \put(6,4){\oval(4,4)[t]}
          \put(10,4){\oval(4,4)[t]}
          \put(12,4){\line(0,1){10}}
          \put(16,4){\line(0,1){10}}
    \end{picture}}}
    \put(11,40){\vector(0,1){65}}
    \put(101,40){\vector(0,1){65}}
    \put(56,135){\vector(0,1){10}}
    \put(22,135){\vector(2,1){20}}
    \put(90,135){\vector(-2,1){20}}
    \put(16,40){\vector(1,2){12}}
    \put(96,40){\vector(-1,2){12}}
    \put(43,95){\vector(1,2){5}}
    \put(67,95){\vector(-1,2){5}}
    \put(47,52){\vector(-1,2){6}}
    \put(61,52){\vector(1,2){6}}
     \put(65,22){\beginpicture\setcoordinatesystem units <.8mm,.8mm>
      \setquadratic
      \plot 0 0  23 12  31 30 /
      \endpicture}
     \put(96,52){\vector(0,1){53}}
     \put(47,22){\beginpicture\setcoordinatesystem units <.8mm,.8mm>
      \setquadratic
      \plot 0 0  -23 12  -31 30 /
      \endpicture}
     \put(16,52){\vector(0,1){53}}
  \end{picture}
  $$
  \centerline{{\it Example:} Hasse diagram for 
    $\mathcal D_{\alpha,\gamma}^\beta$ where 
    $\alpha=(2211),\beta=(433221),\gamma=(32211)$}
\end{figure}

\subsection{Sequences of sources and targets}

Formally, an arc diagram is a finite set of arcs and poles in the 
Poincar\'e half plane.  We assume that all end points are natural numbers 
(arranged from right to left) and permit multiple arcs and poles.\smallskip

We call the left end of an~arc the {\bf source} and the right end  the 
{\bf target}. \smallskip

We view a~pole in an~arc diagram as an~arc with source equal to $\infty$ 
and target equal  to the end point of the pole.\smallskip

With an~arc diagram we associate a~chain of pairs of numbers in the following way.\smallskip
 
 Starting from the left side of the diagram if we meet a~target of an~arc $f$ we write
 $(m,n)$, where $n$ is the target of $f$ 
 and $m$ is the source of $f$. If two or more arcs have the same target $n$ 
we arrange the corresponding pairs
 $(m_1,n),\ldots,(m_x,n)$ in such a~way that $m_1\leq m_2\leq \ldots\leq m_x$.
 
 \begin{ex}
 With the following diagram 
 $$
  
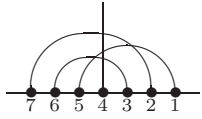
\begin{picture}(32,15)(0,3)
\put(0,4){\line(1,0){32}}
\multiput(3,3)(4,0)7{$\bullet$}
\put(14,4){\oval(20,20)[t]}
\put(20,4){\oval(16,16)[t]}
\put(14,4){\oval(12,12)[t]}
\put(16,4){\line(0,1){15}}
\put(4,1){\makebox[0mm]{$\scriptstyle 7$}}
\put(8,1){\makebox[0mm]{$\scriptstyle 6$}}
\put(12,1){\makebox[0mm]{$\scriptstyle 5$}}
\put(16,1){\makebox[0mm]{$\scriptstyle 4$}}
\put(20,1){\makebox[0mm]{$\scriptstyle 3$}}
\put(24,1){\makebox[0mm]{$\scriptstyle 2$}}
\put(28,1){\makebox[0mm]{$\scriptstyle 1$}}
\end{picture}
  $$
  we associate the chain:
  
  $$(\sinfty,4),(6,3),(7,2),(5,1). $$
 \end{ex}

 
 \begin{lem}
  Two arcs $f$ and $g$, with corresponding chain of pairs $(m,n),(k,r)$
  have a~crossing if and only if $r<n<k<m$.
 \label{lem-order}\end{lem}

 \begin{proof}
  Straightforward.
 \end{proof}
 
 \begin{cor}
  Let $\Delta$ be an~arc diagram with corresponding sequence of sources and targets:
  $$(m_1,n_1),\ldots,(m_x,n_x).$$
  If $m_i>n_j$, for all $i,j\in\{1,\ldots,x\}$, then the following conditions are equivalent
  \begin{enumerate}
   \item $\Delta$ is dominant, i.e.\ there are no crossings in $\Delta$;
   \item $m_1\leq m_2\leq\ldots\leq m_x$.
  \end{enumerate}

 \end{cor}

\subsection{The Bruhat order}

Fix an~LR-tableau $\Gamma$ of type $(\alpha,\beta,\gamma)$ with $\alpha_1\leq 2$. 
Let $x$ denote the number of boxes with entry $1$.
Assume that $\Gamma$ satisfies the following conditions:
\begin{enumerate}
 \item The number of boxes with entry $2$ is equal to $x$ or to $x-1$,
   i.e.\ there is at most one pole in the corresponding arc diagram.
 \item In any row 
   of $\Gamma$ there is at most one non-empty box, i.e.\
   at each point in the corresponding arc diagram, there is at most 
   one arc or pole.
 \item If $j$ is the number of a row with entry $2$ 
   and $i$ is the number of a row with entry $1$, then $j>i$, i.e.\ 
   each starting point of an arc is on the left of every end point 
   of an arc or pole.
\end{enumerate}

We prove that in this case the poset $\mathcal{D}_\Gamma$ 
is related to the Bruhat order of a~symmetric group.

\begin{lem}
 Let $\Gamma$ be an~LR-tableau satisfying the conditions {\rm 1--3}. 
 There is a~bijection between the set $\mathcal{D}_\Gamma$
 and the set $S_x$ of all permutations of $x$ elements.
\label{lem-bijection}\end{lem}

\begin{proof}
Let $\Delta\in\mathcal{D}_\Gamma$.  
Since $\Gamma$ satisfies condition 3,
the corresponding sequence of sources and targets of $\Delta$:
  $$(m_1,n_1),\ldots,(m_x,n_x)$$ 
  is such that $m_i>n_j$, for all $i,j\in\{1,\ldots,x\}$.  
  Moreover, the numbers $m_1,n_1,\ldots,\allowbreak m_x,n_x$ are pairwise different,
  because $\Gamma$ satisfies conditions 1 and 2.
  
  With the sequence $$(m_1,n_1),\ldots,(m_x,n_x)$$ we associate
  the permutation
    $$ (m_1,m_2,\ldots,m_x)=\left(
    \begin{array}{cccc}
      m_{\sigma(1)}&m_{\sigma(2)}&\ldots&m_{\sigma(x)}\\
      m_1&m_2&\ldots&m_x
    \end{array}
    \right),$$
    where $\sigma$ is a~permutation such that $m_{\sigma(1)}<m_{\sigma(2)}<\ldots<m_{\sigma(x)}$.
It is easy to see that this association defines the required bijection.
\end{proof}

\begin{thm} Let $\Gamma$ be an~LR-tableau satisfying  conditions {\rm 1--3}. 
 The poset $(\mathcal{D}_\Gamma,\leq_{\rm arc})$ is isomorphic to the Bruhat order on $S_x$.
\label{thm-bruhat}\end{thm}

\begin{proof}
Let $\Delta$ be an~arc diagram with corresponding sequence of sources and targets:
  $$(m_1,n_1),\ldots,(m_x,n_x).$$
Note that we can do the move (A) or (B) if and only if there exists a~permutation $(m_i,m_j)$ that is an~inversion
(i.e. $i<j$ but $m_i>m_j$).
Therefore the moves of types (A) and (B) in the arc diagram correspond (under the bijection described
in the proof of Lemma \ref{lem-bijection}) to the inversions in $S_x$. 
By \cite[Definition 7.16]{bona}, the  Bruhat order on $S_x$ is generated by inversions.
We are done.
\end{proof}

\begin{cor}\label{cor-saturated-chains}
 Let $\Gamma$ be an~LR-tableau satisfying conditions {\rm 1--3}.
 In the poset $\mathcal{D}_\Gamma$ all saturated chains have the same length.
\end{cor}

\begin{proof}
 By \cite[Proposition 7.18]{bona}, in the Bruhat order on $S_x$ all saturated chains have the same length.
 Therefore, by Theorem \ref{thm-bruhat}, we are done.
\end{proof}

\subsection{Saturated chains}

Let $(\alpha,\beta,\gamma)$ be a~triple of partitions and let $\Gamma$ be an~LR-tableau 
of type $(\alpha,\beta,\gamma)$. 
Let $\Delta,\Delta'$ be elements of the poset $\mathcal{D}_\Gamma$ or of the poset 
$\mathcal{D}_{\alpha,\gamma}^\beta$. We write
\begin{itemize}
 \item $\Delta<_{\rm arc} \Delta'$, if $\Delta\neq \Delta'$ and $\Delta\leq_{\rm arc}\Delta'$,
 \item $\Delta\to\Delta'$, if $\Delta<_{\rm arc}\Delta'$ and there is
no $\Delta''$ such that $\Delta<_{\rm arc}\Delta''<_{\rm arc}\Delta'$. 
\end{itemize}

\begin{lem}
 Let $\Gamma$ be an~LR-tableau such that
in any row of $\Gamma$ there is at most one non-empty box.
 Let $\Delta,\Delta'\in\mathcal{D}_\Gamma$. Then $\Delta'\to\Delta$
 if and only if $\Delta'$ is obtained from $\Delta$ by a~single move of type
 (A) or (B) that reduces the number of crossings by one in the corresponding
 arc diagrams.
\label{lem-byone-1-non-empty}\end{lem}

\begin{proof}
 If $\Delta'$ is obtained from $\Delta$ by a~single move of type
 (A) or (B) that reduces the number of crossings by one in the corresponding
 arc diagrams, then obviously $\Delta'\to\Delta$.
 
 Assume that $\Delta'\to\Delta$. It is obvious that $\Delta'$ is obtained from $\Delta$ 
by a~single move of type (A) or (B) that changes $(m_i,n_i), (m_j,n_j)$
 into $(m_j,n_i), (m_i,n_j)$, where $n_j<n_i<m_j<m_i$. Let $\Delta$ and $\Delta'$ be the arc diagrams associated
 with $\Delta$ and $\Delta'$, respectively. Assume that this move reduces the number of crossings
 by more than one. It follows that in $\Delta$ there exists an~arc $(m_k,n_k)$ such that
 the arcs $(m_i,n_i)$, $(m_j,n_j)$ and $(m_k,n_k)$ in $\Delta$
have at least $2$ more crossings than 
  the arcs $(m_i,n_j)$, $(m_j,n_i)$ and $(m_k,n_k)$ in $\Delta'$. Assume that $(m_k,n_k)$ and $(m_i,n_i)$
 have a~crossing in $\Delta$. 
 We consider two cases: $n_k<n_i<m_k<m_i$ and $n_i<n_k<m_i<m_k$.
 
 If $n_k<n_i<m_k<m_i$, then we have the following possibilities:
 \begin{itemize}
  \item $n_k<n_j<n_i<m_k<m_j<m_i$, there are $3$ crossings in $\Delta$ and $2$ crossings in $\Delta'$;
  \item $n_k<n_j<n_i<m_j<m_k<m_i$, there are $2$ crossings in $\Delta$ and $1$ crossing in $\Delta'$;
  \item $n_j<n_k<n_i<m_k<m_j<m_i$, there are $2$ crossings in $\Delta$ and $1$ crossing in $\Delta'$;
  \item $n_j<n_k<n_i<m_j<m_k<m_i$. there are $3$ crossings in $\Delta$ and no crossing in $\Delta'$.
 \end{itemize}
Note that only in the last case the number of crossings is reduced by more than one. But in this
case we can obtain $\Delta'$ from $\Delta$ by the following sequence of moves: in $\Delta$ we have:
$(m_i,n_i), (m_k,n_k), (m_j,n_j)$; we resolve the crossing $(m_k,n_k), (m_j,n_j)$ and get
$(m_i,n_i), (m_j,n_k), (m_k,n_j)$; then we resolve the crossing $(m_i,n_i), (m_j,n_k)$ and get
$(m_j,n_i), (m_i,n_k), (m_k,n_j)$; finally resolving the last crossing $(m_i,n_k), (m_k,n_j)$
we get $(m_j,n_i), (m_k,n_k), (m_i,n_j)$ in $\Delta'$. It contradicts the assumption that  $\Delta'\to\Delta$.

In the remaining cases the proof is analogous.
\end{proof}

\begin{cor}
 Let $\Gamma$ be an~LR-tableau such that
in any row of $\Gamma$ there is at most one non-empty box.
In the poset $\mathcal{D}_\Gamma$ all saturated chains have the same length.
\label{cor-sat-LR}\end{cor}

\begin{proof}
From Lemma \ref{lem-byone-1-non-empty} it follows that any saturated chain has 
  length equal to the number of crossings in the unique arc-maximal diagram
  in $\mathcal{D}_\Gamma$. 
\end{proof}

\begin{lem}
 Let $(\alpha,\beta,\gamma)$ be a~triple of partitions
 such that $\beta'_i\leq \gamma'_i+1$ for all $i$, 
i.e. $\beta\setminus\gamma$ is a vertical strip;
equivalently, there is at most one entry in each row.

\smallskip
 Let $\Delta,\Delta'\in\mathcal{D}_{\alpha,\gamma}^\beta$. Then $\Delta'\to\Delta$
 if and only if $\Delta'$ is obtained from $\Delta$ by a~single move of type
 (A), (B), (C) or (D) that reduces the number of crossings by one in the corresponding
 arc diagrams.
\label{lem-byone-1-non-empty-ABCD}\end{lem}

\begin{proof}
 The proof is analogous to that of Lemma \ref{lem-byone-1-non-empty}.
\end{proof}

\begin{cor}
Let $(\alpha,\beta,\gamma)$ be a~triple of partitions
 such that $\beta'_i\leq \gamma'_i+1$ for all $i$.
 In the poset $\mathcal{D}_{\alpha,\gamma}^\beta$ all saturated chains have the same length.
\label{cor-sat-general}\end{cor}

\begin{proof}
From Lemmata \ref{lem-byone-1-non-empty} and \ref{lem-byone-1-non-empty-ABCD} 
it follows that any saturated chain has
  length equal to the number of crossings in the unique arc-maximal diagram
  in $\mathcal{D}_{\alpha,\gamma}^\beta$. 
\end{proof}

\begin{rem}
 Note that in the poset $\mathcal D_{2211,32211}^{433221}$ 
given in Example \ref{figure2} on page \pageref{figure2},
 there exist saturated chains of different length. Therefore the assumption
 $\beta'_i\leq \gamma'_i+1$ for all $i$, 
 in Corollary~\ref{cor-sat-general} is necessary.
 Moreover the diagrams $\Delta_{43}^0$, $\Delta_{43}^1$, $\Delta_{43}^{2a}$, 
$\Delta_{43}^{2b}$  and $\Delta_{43}^3$ form the poset  $\mathcal{D}_\Gamma$ for 
\raisebox{-4mm}{\phantom .} 
$\Gamma=         \begin{picture}(18,12)(0,3)
      \multiput(0,9)(3,0)5{\smbox}
      \put(15,9){\numbox{1}}
      \multiput(0,6)(3,0)3{\smbox}
      \put(12,6){\numbox{1}}
      \put(9,6){\numbox{1}}
      \put(0,3){\smbox}
      \put(3,3){\numbox{1}}
      \put(6,3){\numbox{2}}
      \put(0,0){\numbox{2}}
    \end{picture}=\Gamma_{43}$. 
 Note that in this poset there exist saturated chains having different length. 
 Therefore the assumption
 that in any row of $\Gamma$ there is at most one non-empty box, given in Corollary \ref{cor-sat-LR}, is necessary.
\end{rem}

 \section{An algorithmic approach}\label{section-bubble-sort}

We describe an~algorithm which uses moves of type (A) and (B) to
transform an arbitrary arc diagram to the dominant one. We identify arc diagrams with the corresponding
sequences of sources and targets.\smallskip

\subsection{The bubble sort}

 \begin{alg} {\bf The classical bubble sort.}
 
  {\bf Input:}  A~sequence of sources and targets: 
  $$(m_1,n_1),\ldots,(m_x,n_x).$$
  
  {\bf Output:} The sequence 
  $$(m_{\sigma(1)},n_1),\ldots,(m_{\sigma(x)},n_x),$$
  where $m_{\sigma(1)}\leq m_{\sigma(2)}\leq \ldots\leq m_{\sigma(x)}$ and
  $\sigma$ is a~permutation of the set $\{1,\ldots,x\}$.
 
\medskip
 {\bf Description of the algorithm:}
 
  \begin{enumerate}  
  \item $y:=x$
  \item repeat 
   \begin{enumerate}
    \item for all $i=1,\ldots,y-1$ do
      \begin{enumerate}
       \item[(i)] if $m_{i+1}<m_i$, then exchange sources $m_i$ with $m_{i+1}$, 
       i.e. instead of $(m_i,n_i),(m_{i+1},n_{i+1})$ we get $(m_{i+1},n_i),(m_i,n_{i+1})$ 
       (we do the move of type (A) or (B) in the corresponding arc diagram)
      \end{enumerate}
      \item $y:=y-1$
   \end{enumerate}
  until $y=1$.
  \end{enumerate}

 \end{alg}


\begin{lem}
  Let $\Delta$ be an~arc diagram with corresponding sequence of sources and targets
  $$(m_1,n_1),\ldots,(m_x,n_x)$$
  such that $m_i>n_j$, for all $i,j\in\{1,\ldots,x\}$. If $\Delta$ has no multiple
  poles, then every move in (i) in the algorithm reduces the number of crossings by one.
\label{lem-by-one}\end{lem}

\begin{proof}  If we have a~crossing of the arcs $(m_i,n_i)$ and $(m_{i+1},n_{i+1})$,
 then $m_i>m_{i+1}$. After the exchange in (i) we get a~new diagram $\Delta'$ 
 with arcs $(m_{i+1},n_i)$ and $(m_i,n_{i+1})$ that have no crossing. 
 
 If there is a~crossing of arcs $(m_i,n_i)$ and $(m_k,n_k)$ in $\Delta$, where $k\neq i+1$,
 then $n_k<n_i<m_k<m_i$ or $n_i<n_k<m_i<m_k$. Consider the case $n_k<n_i<m_k<m_i$. Note
 that $n_k<n_{i+1}<n_i<m_k<m_i$ and therefore $n_k<n_{i+1}<m_k<m_i$. It follows that there
 is a~crossing of $(m_k,n_k)$ and $(m_i,n_{i+1})$ in $\Delta'$. In the case $n_i<n_k<m_i<m_k$,
 we have $n_{i+1}<n_i<n_k<m_i<m_k$ and a~crossing of $(m_k,n_k)$ and $(m_i,n_{i+1})$ in $\Delta'$.
 
 If there is a~crossing of arcs $(m_{i+1},n_{i+1})$ and $(m_k,n_k)$ in $\Delta$, where $k\neq i$,
 then $n_k<n_{i+1}<m_k<m_{i+1}$ or $n_{i+1}<n_k<m_{i+1}<m_k$. Consider the case $n_k<n_{i+1}<m_k<m_{i+1}$. Note
 that $n_k<n_{i+1}<n_i<m_k<m_{i+1}$ and therefore $n_k<n_i<m_k<m_{i+1}$. It follows that there
 is a~crossing of $(m_k,n_k)$ and $(m_{i+1},n_i)$ in $\Delta'$. In the case $n_{i+1}<n_k<m_{i+1}<m_k$,
 we have $n_{i+1}<n_i<n_k<m_{i+1}<m_k$ and a~crossing of $(m_k,n_k)$ and $(m_{i+1},n_i)$ in $\Delta'$.
 
 If there is a~crossing of arcs $(m_i,n_{i+1})$ and $(m_k,n_k)$ in $\Delta'$, where $k\neq j$,
 then $n_k<n_{i+1}<m_k<m_i$ or $n_{i+1}<n_k<m_i<m_k$. Consider the case $n_k<n_{i+1}<m_k<m_i$. Note
 that $n_k<n_{i+1}<n_i<m_k<m_i$ and therefore $n_k<n_i<m_k<m_i$. It follows that there
 is a~crossing of $(m_k,n_k)$ and $(m_i,n_i)$ in $\Delta$. In the case $n_{i+1}<n_k<m_i<m_k$,
 we have $n_{i+1}<n_i<n_k<m_i<m_k$ and a~crossing of $(m_k,n_k)$ and $(m_i,n_i)$ in $\Delta$.
 
  If there is a~crossing of arcs $(m_{i+1},n_i)$ and $(m_k,n_k)$ in $\Delta'$, where $k\neq j$,
 then $n_k<n_i<m_k<m_{i+1}$ or $n_i<n_k<m_{i+1}<m_k$. Consider the case $n_k<n_i<m_k<m_{i+1}$. Note
 that $n_k<n_{i+1}<n_i<m_k<m_{i+1}$ and therefore $n_k<n_{i+1}<m_k<m_{i+1}$. It follows that there
 is a~crossing of $(m_k,n_k)$ and $(m_{i+1},n_{i+1})$ in $\Delta$. In the case $n_i<n_k<m_{i+1}<m_k$,
 we have $n_{i+1}<n_i<n_k<m_{i+1}<m_k$ and a~crossing of $(m_k,n_k)$ and $(m_{i+1},n_{i+1})$ in $\Delta$.
\end{proof}

 \begin{ex}
  Set $z=1$, $y=x=4$ and do the loop (a). We apply our algorithm to the sequence
  $$(\sinfty,4),(6,3),(7,2),(5,1). $$
  For $i=1$, we compare $(\sinfty,4)$ with $(6,3)$. We have to exchange sources, and we get:
  $$(6,4),(\sinfty,3),(7,2),(5,1). $$
  For $i=2$, we check $(\sinfty,3)$ and $(7,2)$. Since $\sinfty>7>3>2$ we exchange sources
  and get:
  $$(6,4),(7,3),(\sinfty,2),(5,1). $$
  For $i=3$, we compare $(\sinfty,2)$ with $(5,1)$ and (after suitable exchange) we get:
  $$(6,4),(7,3),(5,2),(\sinfty,1). $$
  Now we put $y=3$ and start the second run of the loop (a).
  For $i=1$, we compare $(6,4)$ and $(7,3)$. They are in the proper positions.
  For $i=2$, we compare $(7,3)$ and $(5,2)$. We have to exchange sources, and we get:
  $$(6,4),(5,3),(7,2),(\sinfty,1). $$
  We put $y=2$ and start the third run of the loop (a).
  If $i=1$, we check $(6,4)$ and $(5,3)$, we exchange sources and get:
  $$(5,4),(6,3),(7,2),(\sinfty,1). $$
  We got the arc-minimal diagram.
 \end{ex}

 \begin{alg} {\bf Extended bubble sort.}
 Let $\Delta$ be an~arc diagram with corresponding sequence of sources and targets
  $$(m_1,n_1),\ldots,(m_x,n_x)$$
  
  {\bf Input:}  A~sequence of sources and targets: 
  $$(m_1,n_1),\ldots,(m_x,n_x).$$
  
  {\bf Output:} The sequence 
  $$(m_{\sigma(1)},n_1),\ldots,(m_{\sigma(x)},n_x),$$
  where $m_{\sigma(1)}\leq m_{\sigma(2)}\leq \ldots\leq m_{\sigma(x)}$ and
  the corresponding arc diagram has no crossings.
  
  \pagebreak[1]
 {\bf Description of the algorithm:}

 \smallskip
  Repeat (1)-(8) until there is no crossing in $\Delta$:
  \begin{enumerate}
   \item fix $j$ such that $m_j$ is minimal in the set $\{m_1,\ldots,m_x\}$;
   \item consider the sequence $(m_k,n_k),\ldots,(m_x,n_x)$, where $k$ is such that
     $n_k$ is the maximal element in $\{n_1,\ldots,n_x\}$ that is less than $m_j$;
    \item apply to this sequence the bubble sort algorithm;
    \item note that we got the sequence $(m_k,n_k),\ldots,(m_x,n_x)$, where $m_k=m_j$;
       note also that the arc $(m_k,n_k)$ has no crossings;
    \item remove $(m_k,n_k)$ from the sequence $(m_1,n_1),\ldots,(m_x,n_x)$;
    \item set $m_{i-1}=m_i$ and $n_{i-1}=n_i$ for all $i=k+1,\ldots,x$;
    \item set $x=x-1$;
    \item come back to (1);
  \end{enumerate}

 \end{alg}

\begin{ex}
 Consider an~arc diagram with the following sequence of sources and targets:

$$
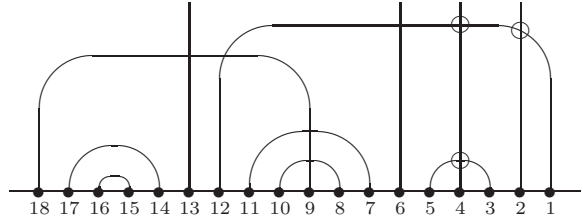
\setlength\unitlength{1mm}
\begin{picture}(76,32)(0,-5)
  \put(0,0){\phantom m}
  \put(76,30){\phantom m}
        \put(0,4){\line(1,0){76}}
        \multiput(3,3)(4,0){18}{$\bullet$}
        \put(4,1){\sb{18}}
        \put(8,1){\sb{17}}
        \put(12,1){\sb{16}}
        \put(16,1){\sb{15}}
        \put(20,1){\sb{14}}
        \put(24,1){\sb{13}}
        \put(28,1){\sb{12}}
        \put(32,1){\sb{11}}
        \put(36,1){\sb{10}}
        \put(40,1){\sb9}
        \put(44,1){\sb8}
        \put(48,1){\sb7}
        \put(52,1){\sb6}
        \put(56,1){\sb5}
        \put(60,1){\sb4}
        \put(64,1){\sb3}
        \put(68,1){\sb2}
        \put(72,1){\sb1}
        \put(14,4){\oval(4,4)[t]}
        \put(14,4){\oval(12,12)[t]}
        \put(24,4){\line(0,1){25}}
        \put(22,4){\oval(36,36)[t]}
        \put(40,4){\oval(8,8)[t]}
        \put(40,4){\oval(16,16)[t]}
        \put(52,4){\line(0,1){25}}
        \put(60,4){\line(0,1){25}}
        \put(60,4){\oval(8,8)[t]}
        \put(68,4){\line(0,1){25}}
        \put(50,4){\oval(44,44)[t]}
        \put(60,7.5){\sb\bigcirc}
        \put(68,24.8){\sb\bigcirc}
        \put(60,25.5){\sb\bigcirc}
        \put(38,-5){\makebox(0,0){$(16,15),(17,14),(\sinfty,13),(18,9),(10,8),(11,7),(\sinfty,6),(\sinfty,4),(5,3),(\sinfty,2),(12,1)$}} 
    \end{picture}
$$

 In the step (1) of the algorithm we have $j=9$ and $m_j=5$. Moreover $k=8$ and $n_k=4$.
 
 We apply the bubble sort to the four arcs and poles ending at 4, 3, 2, and 1:
 
 $$ (\sinfty,4), (5,3), (\sinfty,2), (12,1) $$
 
In three steps, the algorithm removes the three encircled intersections.
We get:
 
 $$ (5,4), (12,3), (\sinfty,2), (\sinfty,1). $$
 
 We have the following arc diagram:

 $$
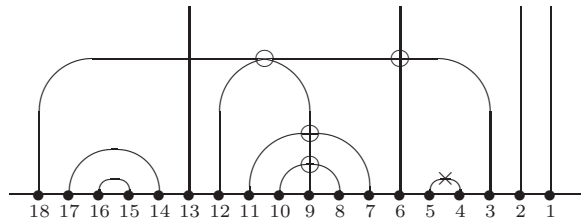
\setlength\unitlength{1mm}\begin{picture}(76,32)(0,-5)
  \put(0,0){\phantom m}
  \put(76,30){\phantom m}
        \put(0,4){\line(1,0){76}}
        \multiput(3,3)(4,0){18}{$\bullet$}
        \put(4,1){\sb{18}}
        \put(8,1){\sb{17}}
        \put(12,1){\sb{16}}
        \put(16,1){\sb{15}}
        \put(20,1){\sb{14}}
        \put(24,1){\sb{13}}
        \put(28,1){\sb{12}}
        \put(32,1){\sb{11}}
        \put(36,1){\sb{10}}
        \put(40,1){\sb9}
        \put(44,1){\sb8}
        \put(48,1){\sb7}
        \put(52,1){\sb6}
        \put(56,1){\sb5}
        \put(60,1){\sb4}
        \put(64,1){\sb3}
        \put(68,1){\sb2}
        \put(72,1){\sb1}
        \put(14,4){\oval(4,4)[t]}
        \put(14,4){\oval(12,12)[t]}
        \put(24,4){\line(0,1){25}}
        \put(22,4){\oval(36,36)[t]}
        \put(40,4){\oval(8,8)[t]}
        \put(40,4){\oval(16,16)[t]}
        \put(52,4){\line(0,1){25}}
        \put(68,4){\line(0,1){25}}
        \put(58,5.2){\sb{\textstyle\times}}
        \put(58,4){\oval(4,4)[t]}
        \put(72,4){\line(0,1){25}}
        \put(46,4){\oval(36,36)[t]}
        \put(40,7.3){\sb\bigcirc}
        \put(52,21.4){\sb\bigcirc}
        \put(40,11.5){\sb\bigcirc}
        \put(34,21.4){\sb\bigcirc}
        \put(38,-5){\makebox(0,0){$(16,15),(17,14),(\sinfty,13),(18,9),(10,8),(11,7),(\sinfty,6),(5,4), (12,3), (\sinfty,2), (\sinfty,1)$}}
    \end{picture}
$$
 
 We remove the arc $(5,4)$ (labelled by an $\times$)
and apply (1)-(8) to the sequence:
 
 $$(16,15),(17,14),(\sinfty,13),(18,9),(10,8),(11,7),(\sinfty,6), (12,3), (\sinfty,2), (\sinfty,1). $$
 
 Now we have $j=5$, $m_j=10$, $k=4$ and $n_k=9$. We apply the bubble sort to
the arcs and poles ending at or on the right of 9:
 
$$ (18,9),(10,8),(11,7),(\sinfty,6), (12,3), (\sinfty,2), (\sinfty,1) $$
 
In four steps, the algorithm removes the four encircled intersections. We get 
 
 $$ (10,9),(11,8),(12,7),(18,6), (\sinfty,3), (\sinfty,2), (\sinfty,1). $$
 
 Our sequence has the form:
 
$$
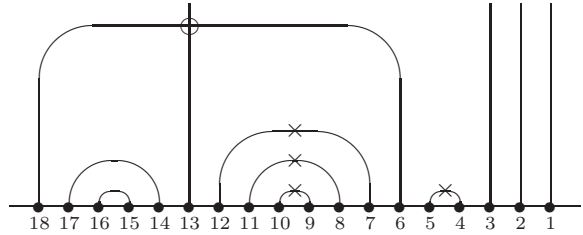
\setlength\unitlength{1mm}\begin{picture}(76,32)(0,-5)
  \put(0,0){\phantom m}
  \put(76,30){\phantom m}
        \put(0,4){\line(1,0){76}}
        \multiput(3,3)(4,0){18}{$\bullet$}
        \put(4,1){\sb{18}}
        \put(8,1){\sb{17}}
        \put(12,1){\sb{16}}
        \put(16,1){\sb{15}}
        \put(20,1){\sb{14}}
        \put(24,1){\sb{13}}
        \put(28,1){\sb{12}}
        \put(32,1){\sb{11}}
        \put(36,1){\sb{10}}
        \put(40,1){\sb9}
        \put(44,1){\sb8}
        \put(48,1){\sb7}
        \put(52,1){\sb6}
        \put(56,1){\sb5}
        \put(60,1){\sb4}
        \put(64,1){\sb3}
        \put(68,1){\sb2}
        \put(72,1){\sb1}
        \put(14,4){\oval(4,4)[t]}
        \put(14,4){\oval(12,12)[t]}
        \put(24,4){\line(0,1){27}}
        \put(28,4){\oval(48,48)[t]}
        \put(38,4){\oval(4,4)[t]}
        \put(38,4){\oval(12,12)[t]}
        \put(64,4){\line(0,1){27}}
        \put(68,4){\line(0,1){27}}
        \put(58,4){\oval(4,4)[t]}
        \put(72,4){\line(0,1){27}}
        \put(38,4){\oval(20,20)[t]}
        \put(24,27.3){\sb\bigcirc}
        \put(58,5.2){\sb{\textstyle\times}}
        \put(38,5.2){\sb{\textstyle\times}}
        \put(38,9.2){\sb{\textstyle\times}}
        \put(38,13.2){\sb{\textstyle\times}}
        \put(38,-5){\makebox(0,0){$(16,15),(17,14),(\sinfty,13),(10,9),(11,8),(12,7),(18,6), (\sinfty,3), (\sinfty,2), (\sinfty,1)$}}
    \end{picture}
$$

 We can remove arcs $(10,9),(11,8),(12,7)$. So we apply (1)-(8) to
 
 $$(16,15),(17,14),(\sinfty,13),(18,6), (\sinfty,3), (\sinfty,2), (\sinfty,1). $$
 
 Now $j=1$, $m_j=16$, $k=1$, $n_k=15$ and we apply the bubble sort to the full
sequence:
 
 $$(16,15),(17,14),(\sinfty,13),(18,6), (\sinfty,3), (\sinfty,2), (\sinfty,1) $$
 
The algorithm removes the last intersection in one step. We get
 
 $$(16,15),(17,14),(18,13),(\sinfty,6), (\sinfty,3), (\sinfty,2), (\sinfty,1). $$
 
 Our arc diagram has no crossings. The algorithm terminates with the output:
$$
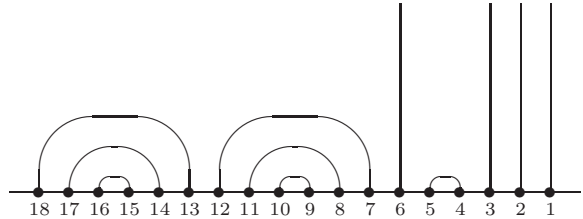
\setlength\unitlength{1mm}\begin{picture}(76,32)(0,-5)
  \put(0,0){\phantom m}
  \put(76,30){\phantom m}
        \put(0,4){\line(1,0){76}}
        \multiput(3,3)(4,0){18}{$\bullet$}
        \put(4,1){\sb{18}}
        \put(8,1){\sb{17}}
        \put(12,1){\sb{16}}
        \put(16,1){\sb{15}}
        \put(20,1){\sb{14}}
        \put(24,1){\sb{13}}
        \put(28,1){\sb{12}}
        \put(32,1){\sb{11}}
        \put(36,1){\sb{10}}
        \put(40,1){\sb9}
        \put(44,1){\sb8}
        \put(48,1){\sb7}
        \put(52,1){\sb6}
        \put(56,1){\sb5}
        \put(60,1){\sb4}
        \put(64,1){\sb3}
        \put(68,1){\sb2}
        \put(72,1){\sb1}
        \put(14,4){\oval(4,4)[t]}
        \put(14,4){\oval(12,12)[t]}
        \put(52,4){\line(0,1){25}}
        \put(14,4){\oval(20,20)[t]}
        \put(38,4){\oval(4,4)[t]}
        \put(38,4){\oval(12,12)[t]}
        \put(64,4){\line(0,1){25}}
        \put(68,4){\line(0,1){25}}
        \put(58,4){\oval(4,4)[t]}
        \put(72,4){\line(0,1){25}}
        \put(38,4){\oval(20,20)[t]}
        \put(38,-5){\makebox(0,0){$(16,15),(17,14),(18,13),(10,9),(11,8),(12,7),(\sinfty,6),(5,4), (\sinfty,3), (\sinfty,2), (\sinfty,1)$}}
    \end{picture}
$$
 \end{ex}
 
 \begin{lem}
  Let $\Delta$ be an~arc diagram with corresponding sequence of sources and targets
  $$(m_1,n_1),\ldots,(m_x,n_x)$$
  and let $\Delta'$ be the dominant arc diagram of the same LR-type.
  If $\Delta$ has no multiple
  poles, then there is a~sequence of moves that reduce $\Delta$ to $\Delta'$ such that
  after every move the number of crossings is decreasing by one.
\label{lem-by-one-general}\end{lem}

\begin{proof}
 It follows from the algorithms and Lemma \ref{lem-by-one}.
\end{proof}

\section{Three excursions}\label{section-excursions}

\subsection{Some projective varieties}

We show that projective spaces and Grassmann varieties occur as 
quotients of diagram varieties.

\medskip

Denote by $\mathbb D_{\alpha,\gamma}^\beta(k)$ the subset of the Grassmann variety
$\mathbb{G}(|\alpha|,k^{|\beta|})$ consisting of all submodules $U\subseteq N_{\beta}(k)$
 such that $U\cong N_{\alpha}(k)$ and $N_{\beta}(k)/U\cong N_{\gamma}(k)$.
 By $\mathbb V_{\alpha,\gamma}^\beta(k)$ denote the subset of the affine
 variety $\mathbb H_\alpha^\beta(k)=\Hom_k(N_\alpha(k),N_\beta(k))$
 consisting of all monomorphisms $f:N_\alpha(k)\to N_\beta(k)$
 with $\Coker\, f\cong N(\gamma)$.
 The group $\Aut (N_{\alpha}(k))^{op}$ acts freely on $\mathbb V_{\alpha,\gamma}^\beta(k)$
 in the following way. For $\sigma \in \Aut (N_{\alpha}(k))^{op}$ and $f\in \mathbb V_{\alpha,\gamma}^\beta(k)$
 we set $$ \sigma\cdot f=f\circ \sigma.$$
 The map $$ F:\mathbb V_{\alpha,\gamma}^\beta(k)\to \mathbb D_{\alpha,\gamma}^\beta(k) $$
 defined by $F(f)=(\Im\, f\subseteq N(\beta))$ is polynomial and its fibers are isomorphic to
 $\Aut (N_{\alpha}(k))^{op}$.
 \begin{rem} Let $(1^m)$ denote the partition $(1,\ldots,1)$ with $m$ parts.
  \begin{enumerate}
  \item Projective spaces are arc diagram varieties as
    $$\mathbb P(k^m)=\mathbb D_{(1),(1^{m-1})}^{(1^m)}(k) \quad\text{for}\quad m\in\mathbb N.$$
    Note that $\mathbb V_{(1),(1^{m-1})}^{(1^m)}(k)=k^m\setminus\{0\}$.
  \item Grassmann varieties can be realized as
    $$\mathbb G(\ell, k^m)=\mathbb D_{(1^\ell),(1^{m-\ell})}^{(1^m)}(k) \quad\text{for}\quad \ell,m\in\mathbb N, \ell\leq m.$$
    The variety $\mathbb V_{(1^\ell),(1^{m-\ell})}^{(1^m)}(k)$ consists of all $l\times m$ matrices with
    maximal rank.
  \end{enumerate}
\end{rem}

For finite fields, the size of the projective varieties is under control:

\medskip
We have  

$$| \mathbb D_{\alpha,\gamma}^\beta(\mathbb F_p)| = g_{\alpha,\gamma}^\beta(p)$$
and 
$$| \mathbb V_{\alpha,\gamma}^\beta(\mathbb F_p)| = |\Aut N_\alpha(\mathbb F_p)| \cdot | \mathbb D_{\alpha,\gamma}^\beta(\mathbb F_p)|=|\Aut N_\alpha(\mathbb F_p)| \cdot g_{\alpha,\gamma}^\beta(p).$$

\subsection{Degenerations of nilpotent operators}

Classical Hall polynomials allow to investigate geometric
properties of nilpotent operators.

Let $k$ be an~arbitrary algebraically closed field. 
We consider the affine variety $\mathbb{M}_n(k)$
consisting of all $n\times n-$matrices with coefficients in $k$. 
On $\mathbb{M}_n(k)$ we consider the Zariski topology and on all subsets of
$\mathbb{M}_n(k)$ we work with the induced topology. By $\mathbb{M}_n^0(k)$ denote
the closed subset of $\mathbb{M}_n(k)$ consisting of nilpotent matrices. 
The general linear group $Gl_n(k)$ acts on $\mathbb{M}_n^0(k)$
via conjugation:  $g\cdot A=gAg^{-1}$. The orbits of this action correspond bijectively
to isomorphism classes of objects in $\mathcal{N}(k,n)$, where  $\mathcal{N}(k,n)$
is the full subcategory of $\mathcal{N}(k)$ consisting of all objects $N_\alpha=N_\alpha(k)$
such that $\dim_kN_\alpha=n$. Denote by $G_\alpha=G_\alpha(k)$ the orbit of
$N_\alpha$ in $\mathbb{M}_n^0(k)$.

\begin{defin}
Let $N_\alpha$ and $N_\beta$ be objects 
in $\mathcal{N}(k,n)$.
The relation $N_\alpha \leq_{\rm deg} N_\beta$ 
holds if $G_\beta(k) \subseteq \overline{G_\alpha(k)}$ in $\mathbb{M}_n^0(k)$,
where $\overline{G(k)}$ is the closure of $G(k)$. 
\end{defin}

The following theorem is well known (see \cite[I.3]{kraft}) 

\begin{thm} 
Let $N_\alpha$ and $N_\beta$ be objects 
in $\mathcal{N}(k,n)$.
The relation $N_\alpha\leq_{\rm deg} N_\beta$ holds
if and only if  $\sum_{i=1}^m\alpha'_i\leq \sum_{i=1}^m\beta'_i$
for all $m\in \mathbb{N}$, where $\alpha'$ denotes the conjugate partition of $\alpha$.
\end{thm}

Let $\alpha$ and $\beta$ be partitions of $n$. We write $\alpha\mapsto_{\rm box}
\beta$ if there exists $i<j$ such that $\alpha_i=\beta_i+1$, $\alpha_j=\beta_j-1$
and $\alpha_k=\beta_k$ for $k\neq i,j$. We define the {\bf box order}  $\alpha\leq_{\rm box}
\beta$ to be the partial order generated by all moves $\mapsto_{\rm box}$.

If we look at  Young diagrams, the box order is generated 
by a~sequence of moves of type (going up with a~box):

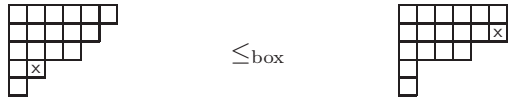
\vspace{-4mm}
$$\begin{picture}(18,12)(0,6)
\multiput(0,12)(3,0)5{\smbox}
\put(15,12){\smbox}
\multiput(0,9)(3,0)4{\smbox}
\put(12,9){\smbox}
\multiput(0,6)(3,0)2{\smbox}
\multiput(6,6)(3,0)2{\smbox}
\put(0,3){\smbox}
\put(3,3){\numbox{x}}
\put(0,0){\smbox}
\end{picture}
\qquad\qquad \leq_{\rm box} \qquad\qquad
\begin{picture}(18,12)(0,6)
\multiput(0,12)(3,0)5{\smbox}
\put(15,12){\smbox}
\multiput(0,9)(3,0)4{\smbox}
\put(12,9){\smbox}
\multiput(0,6)(3,0)2{\smbox}
\multiput(6,6)(3,0)2{\smbox}
\put(0,3){\smbox}
\put(15,9){\numbox{x}}
\put(0,0){\smbox}
\end{picture} $$
\vspace{5mm}

\begin{thm}
 Let $\alpha$ and $\beta$ be partitions of $n$. Then
  $$ N_\alpha\leq_{\rm deg} N_\beta \;\;\; \mbox{ if and only if }\;\;\; \alpha\leq_{\rm box}\beta. $$
\end{thm}

{\bf Proof.} Let $N_\alpha\leq_{\rm deg} N_\beta$ and $\alpha\neq \beta$. It follows that  
$\sum_{i=1}^m\alpha'_i\leq \sum_{i=1}^m\beta'_i$
for all $m\in \mathbb{N}$. Let $s$ be the minimal natural number such that  
$\sum_{i=1}^s\alpha'_i< \sum_{i=1}^s\beta'_i$. It follows that
$\alpha'_i=\beta'_i$ for all $i=1,\ldots, s-1$ and $\alpha'_s<\beta'_s$. Since 
$\sum_{i=1}^\infty\alpha'_i= \sum_{i=1}^\infty\beta'_i=n$,
there exists $t>s$ such that $\alpha'_t>\beta'_t$. Chose 
$t$ minimal with this property. Let $\gamma$ be the partition such that
$$\gamma'=(\beta'_1,\ldots,\beta'_{s-1},\beta'_s-1,\beta'_{s+1},
\ldots,\beta'_{t-1},\beta'_t+1,\beta'_{t+1},\ldots).$$ 
It is straightforward to check that  $\gamma\leq_{\rm box}\beta$ and
$$\sum_{i=1}^m\alpha'_i\leq \sum_{i=1}^m\gamma'_i \leq\sum_{i=1}^m\beta'_i,$$
for all $m$, and 
$$\sum_{i=1}^s\alpha'_i\leq \sum_{i=1}^s\gamma'_i <\sum_{i=1}^s\beta'_i.$$
Therefore we have $N_\alpha\leq_{\rm deg} N_\gamma$. 
Continuing this procedure we prove that $\alpha\leq_{\rm box}\beta$.

Conversely, assume that $\alpha\leq_{\rm box}\beta$ is given by single ``box move''.
It is easy to prove that $\sum_{i=1}^m\alpha'_i\leq \sum_{i=1}^m\beta'_i$
for all $m\in \mathbb{N}$. Therefore $N_\alpha\leq_{\rm deg} N_\beta$ and we are done. \epv

Combining results presented in \cite[I.3]{kraft} and in \cite{macd}
we can prove the following.

\begin{thm}
 Let
$$ 0\to N_\alpha\to N_\lambda\to N_\beta\to 0, $$
$$ 0\to N_\alpha\to N_\gamma\to N_\beta\to 0 $$
be short exact sequences of $k[T]$-modules.
If $N_\lambda\leq_{\rm deg}N_\gamma$, then
$$ \deg g_{\alpha\beta}^\lambda \leq \deg g_{\alpha\beta}^\gamma.$$
\end{thm}

\begin{proof}
Assume that  $N_\lambda\leq_{\rm deg}N_\gamma$. By \cite[I.3]{kraft}, we obtain for any $m\geq 1$:
$$ \sum_{i=1}^m\lambda'_i \leq \sum_{i=1}^m\gamma'_i.$$
It follows from \cite[Section I, 1.11]{macd} that the following inequality holds 
for any $m\geq 1$:
$$ \sum_{i=1}^m\lambda_i \geq \sum_{i=1}^m\gamma_i.$$
Let $\lambda=(\lambda_1,\ldots,\lambda_k)$ and $\gamma=(\gamma_1,\ldots,\gamma_n)$.
Since $\lambda_1+\ldots+\lambda_k=\gamma_1+\ldots+\gamma_n$ and $ \sum_{i=1}^m\lambda_i \geq \sum_{i=1}^m\gamma_i$,
we have $n\geq k$. We prove that $n(\lambda)\leq n(\gamma)$. Consider the equality
$k\cdot\lambda_1+\ldots+k\cdot\lambda_k=k\cdot\gamma_1+\ldots+k\cdot\gamma_n$ and subtract
inequalities $ \sum_{i=1}^m\lambda_i \geq \sum_{i=1}^m\gamma_i$, for $m=1,\ldots,k$.
We get 
$$\begin{array}{rcl}
n(\lambda)&=&0\cdot\lambda_1+1\cdot\lambda_2+\ldots+(k-1)\cdot\lambda_k \\
&\leq&
0\cdot\gamma_1+1\cdot\lambda_2+\ldots+(k-1)\cdot\gamma_k+k\cdot\gamma_{k+1}+\ldots+k\cdot\gamma_n\leq n(\gamma).   
  \end{array}
$$
This finishes the proof, because
  $$ \deg g_{\alpha\beta}^\lambda =n(\lambda)-n(\alpha)-n(\beta)\leq n(\gamma)-n(\alpha)-n(\beta)=\deg g_{\alpha\beta}^\gamma.$$
\end{proof}

Connections of Young tableaux and partitions 
with degenerations and generic extensions of nilpotent operators
are also studied in \cite{kos2012}.

\subsection{A partial ordering on LR-tableaux}

Let $\Gamma$ be an LR-tableau of type $(\alpha,\beta,\gamma)$. 
Note that the poset structure on $\mathcal D_\Gamma$
is just the restriction of $\mathcal D_{\alpha,\gamma}^\beta$ 
to the arc diagrams in $\Gamma$.  
By ``identifying'' those arc diagrams we obtain a poset structure
$\barD$ on the set of LR-tableaux
of type $(\alpha,\beta,\gamma)$; the relation is given by
$$\Gamma\leq \Gamma' \quad\Longleftrightarrow\quad
  \mathbb V_{\Gamma'}\cap \overline{\mathbb V}_\Gamma \neq 0.$$
We can characterize this order relation in different ways.

\smallskip
\setlength\unitlength{1mm}
Assume that $\alpha_1\leq 2$ holds, then an LR-tableau $\Gamma$
of type $(\alpha,\beta,\gamma)$ is given by three partitions
$\gamma\subset \tilde\gamma \subset \beta$ where 
the {\bf intermediate partition} $\tilde \gamma$ is such
that $\beta\setminus\tilde\gamma$ consists of all boxes $\singlebox 1$
and $\tilde\gamma\setminus\gamma$ consists of all boxes $\singlebox 2$. 

\begin{prop}\label{prop-dbar}
The following assertions are equivalent for LR-tableaux $\Gamma$, $\Gamma'$
of type $(\alpha,\beta,\gamma)$.
\begin{enumerate}
\item $\Gamma\leq \Gamma'$.
\item There exists $\Delta'\in\mathcal D_{\Gamma'}$ such that 
  $\mathbb V_{\Delta'}\subset\overline{\mathbb V}_\Gamma$.
\item There are $\Delta\in\mathcal D_\Gamma$ and $\Delta'\in\mathcal D_{\Gamma'}$
  such that $\Delta\arcleq\Delta'$.
\item The intermediate partitions $\tilde\gamma$ for $\Gamma$ and
  ${\tilde\gamma}'$ for $\Gamma'$ satisfy ${\tilde\gamma}\boxleq {\tilde\gamma}'$.
\end{enumerate}
\end{prop}

\begin{proof}
The equivalence of 1.\ and 2.\ is clear from the definitions.
To see that 2.\ implies 3.\ note that there is an arc diagram $\Delta$
such that $\overline{\mathbb V}_\Gamma=\overline{\mathbb V}_\Delta$;
the converse holds by Theorem~\ref{thm-first-main} and 
since $\mathbb V_\Delta\subset \mathbb V_\Gamma$
implies that $\overline{\mathbb V}_\Delta\subset \overline{\mathbb V}_\Gamma$.

\smallskip
We show that 3.\ implies 4.\  Suppose $\Delta\arcleq \Delta'$,
then there is a sequence of moves which convert $\Delta$ to $\Delta'$.
Note that moves of type (A) or (B) leave the underlying LR-tableau unchanged,
while moves of type (C) or (D) exchange the positions of a box 
$\singlebox 1$ with a box $\singlebox 2$:
If $\Delta\arcleq\Delta'$ then ${\tilde\gamma}\boxleq{\tilde\gamma}'$.

\smallskip
For the converse we assume that the intermediate partitions
$\tilde\gamma$ for $\Gamma$ and ${\tilde\gamma}'$ for $\Gamma'$
satisfy the relation ${\tilde\gamma}\boxleq {\tilde\gamma}'$
and are such that ${\tilde\gamma}$ is obtained from ${\tilde \gamma}'$ by the
move of a single box, say from the $a$-th row up into the $b$-th row.
Let $\Delta'$ be the unique arc diagram of type $\Gamma'$ with the
maximal number of intersections. It follows that there is an arc in $\Delta'$
starting at $a$ which intersects an arc or pole in $\Delta'$ ending at $b$.
The arc move of type (C) or (D) 
which resolves this intersection yields a diagram $\Delta$
of type $\Gamma$.  Thus, $\Delta\arcleq\Delta'$. 
\end{proof}

\begin{prop}
\begin{enumerate}
\item The poset $\barD$ has a unique maximal element,
  it is the LR-tableau given by the unique arc diagram with the 
  maximal number of intersections. Equivalently, it is the 
  LR-tableau of type $(\alpha,\beta,\gamma)$ in which the 
  boxes $\singlebox 2$ are in the largest available rows.
\item The poset $\barD$ has a unique minimal element,
  it is given by the unique LR-tableau that can be refined only
  to arc diagrams with no intersections. 
  Equivalently, it is the LR-tableau of type $(\alpha,\beta,\gamma)$
  in which the boxes $\singlebox 2$ are in the smallest available rows.
\end{enumerate}
\end{prop}

\begin{proof}
The first statement follows from Theorem~\ref{thm-lattice} 
and Proposition~\ref{prop-dbar}.
Namely, if $\Delta$ is the unique arc 
diagram with a maximum number of intersections, then $\mathbb V_\Delta$
is contained in the closure of any other stratum. 
Recall from \cite[Proof of Theorem~5.7]{kossch} that this LR-tableau
is such that the entries $\singlebox2$ are in the largest available rows.

\smallskip
Consider the LR-tableau $\Gamma$ which is such that the entries 
$\singlebox2$ are in the smallest available rows (to obtain $\Gamma$,
proceed rowwise from the top, and put in each row the largest possible
number of $\singlebox2$'s).  Let $\Delta$ be an arc diagram of type $\Gamma$.
It is not possible to resolve any intersection in $\Delta$ 
by arc moves of type (C) or (D)
since each such move lifts a box $\singlebox2$ into a higher row.
Since moves of type (A) and (C), and of type (B) and (D) occur pairwise,
it is not possible to resolve any intersection in $\Delta$ by arc
moves, i.e.\ $\Delta$ has no intersection.
\end{proof}

\begin{ex}
We revisit Example~\ref{figure2} on page~\pageref{figure2}.  Corresponding to the partitions
$\alpha=(2,2,1,1)$, $\beta=(4,3,3,2,2,1)$, $\gamma=(3,2,2,1,1)$ are
the four LR-tableaux $\Gamma_{43}$, $\Gamma_{42}$, $\Gamma_{33}$ and
$\Gamma_{32}$ pictured in Section~\ref{ex-gammas}. 
Here is the Hasse diagram for the partial ordering 
in $\barD$.
$$\begin{picture}(30,30)
\put(0,15){\makebox[0pt]{$\Gamma_{42}$}}
\put(15,0){\makebox[0pt]{$\Gamma_{32}$}}
\put(30,15){\makebox[0pt]{$\Gamma_{33}$}}
\put(15,30){\makebox[0pt]{$\Gamma_{43}$}}
\put(19,6){\vector(1,1){7}}
\put(4,21){\vector(1,1){7}}
\put(11,6){\vector(-1,1){7}}
\put(26,21){\vector(-1,1){7}}
\end{picture}$$
\end{ex}

\bigskip
Address of the authors:
\nopagebreak

\parbox[t]{5.5cm}{\footnotesize\begin{center}
              Faculty of Mathematics\\
              and Computer Science\\
              Nicolaus Copernicus University\\
              ul.\ Chopina 12/18\\
              87-100 Toru\'n, Poland\\[1ex]
              {\tt justus@mat.umk.pl}\end{center}}
\parbox[t]{5.5cm}{\footnotesize\begin{center}
              Department of\\
              Mathematical Sciences\\ 
              Florida Atlantic University\\
              777 Glades Road\\
              Boca Raton, Florida 33431\\[1ex]
              {\tt markus@math.fau.edu}\end{center}}

\end{document}